\documentclass{amsart}

\usepackage{amsmath, amssymb,amsfonts}
\usepackage{mathrsfs}
\usepackage[all]{xy}
\usepackage[colorlinks=true]{hyperref}

\title[Cubic threefolds, hyperk\"ahler manifolds and the complex ball]{Cubic threefolds and hyperk\"ahler manifolds uniformized by the 10-dimensional complex ball}

\author{Samuel Boissi\`ere}
\address{Samuel Boissi\`ere, Universit\'e de Poitiers, 
Laboratoire de Math\'ematiques et Applications, 
 T\'el\'eport 2, Boulevard Marie et Pierre Curie
 BP 30179, 86962 Futuroscope Chasseneuil Cedex, France}
\email{samuel.boissiere@math.univ-poitiers.fr}
\urladdr{http://www-math.sp2mi.univ-poitiers.fr/$\sim$sboissie/}

\author{Chiara Camere}
\address{Chiara Camere, Universit\`a  degli Studi di Milano,
Dipartimento di Matematica,
Via Cesare Saldini 50,
20133 Milano, Italy} 
\email{chiara.camere@unimi.it}
\urladdr{http://www.mat.unimi.it/users/camere/en/index.html}

\author{Alessandra Sarti}
\address{Alessandra Sarti, Universit\'e de Poitiers, 
Laboratoire de Math\'ematiques et Applications, 
 T\'el\'eport 2, Boulevard Marie et Pierre Curie,
 BP 30179, 86962 Futuroscope Chasseneuil Cedex, France}
\email{sarti@math.univ-poitiers.fr}
\urladdr{http://www-math.sp2mi.univ-poitiers.fr/$\sim$sarti/}

\date{\today}

\newcommand{\ie}{{\it i.e. }}
\newcommand{\loccit}{{\it loc. cit.}}
\newcommand{\resp}{{\it resp.\,}}

\newcommand{\apriori}{{\it a priori }}

\newcommand{\IB}{\mathbb B}
\newcommand{\IC}{\mathbb C}

\newcommand{\IP}{\mathbb P}
\newcommand{\IQ}{\mathbb Q}
\newcommand{\IR}{\mathbb R}
\newcommand{\IZ}{\mathbb Z}

\newcommand{\cA}{\mathcal{A}}
\newcommand{\cC}{\mathcal{C}}
\newcommand{\cE}{\mathcal{E}}
\newcommand{\cF}{\mathcal{F}}

\newcommand{\cH}{\mathcal{H}}
\newcommand{\cK}{\mathcal{K}}
\newcommand{\cM}{\mathcal{M}}
\newcommand{\cN}{\mathcal{N}}
\newcommand{\cO}{\mathcal{O}}
\newcommand{\cP}{\mathcal{P}}
\newcommand{\cR}{\mathcal{R}}
\newcommand{\cZ}{\mathcal{Z}}

\newcommand{\kM}{\mathfrak{M}}
\newcommand{\kS}{\mathfrak{S}}

\newcommand{\calC}{\mathscr{C}}
\newcommand{\calH}{\mathscr{H}}

\newcommand{\rF}{{\rm{F}}}

\newcommand{\lra}{\longrightarrow}

\DeclareMathOperator{\Aut}{Aut}
\DeclareMathOperator{\Bir}{Bir}

\DeclareMathOperator{\Span}{Span}
\DeclareMathOperator{\PGL}{PGL}
\DeclareMathOperator{\Orth}{O}

\DeclareMathOperator{\disc}{disc}
\DeclareMathOperator{\rk}{rk}

\DeclareMathOperator{\Res}{Res}
\DeclareMathOperator{\Mon}{Mon}
\DeclareMathOperator{\IntJac}{IJ}
\DeclareMathOperator{\Alb}{Alb}
\DeclareMathOperator{\Pf}{Pf}
\DeclareMathOperator{\Grass}{Grass}
\DeclareMathOperator{\Flag}{Flag}

\newcommand{\proj}[1]{\IP^{#1}} 
\newcommand{\ls}[1]{|{#1}|} 
\newcommand{\coloneqq}{:=}
\newcommand{\eqqcolon}{=:}

\newcommand{\leftquo}[2]{{\left.\raisebox{-.2em}{$#1$}\middle\backslash\raisebox{.2em}{$#2$}\right.}}

\newcommand{\NS}{{\rm NS}}
\newcommand{\HH}{{\rm H}}
\newcommand{\id}{{\rm id}}
\newcommand{\Fano}[1]{\rF(#1)}
\newcommand{\stable}{{\rm s}}

\newcommand{\smooth}{{\rm sm}}
\newcommand{\nodal}{{\rm n}}
\newcommand{\chordal}{{\rm c}}
\newcommand{\cov}{{\rm cov}}
\newcommand{\prim}{{\circ}}
\newcommand{\orth}{{\circ}}

\newcommand{\marked}{{\rm mk}}
\newcommand{\Exc}{{\rm Exc}}
\newcommand{\Hdg}{{\rm Hdg}}

\newcommand{\Pfaff}{\cP\kern-.3em f}

\newcommand{\BCSmodfr}{\cM_{\langle 6\rangle}^{\rho,\xi}}
\newcommand{\BCSmod}{\cN_{\langle 6\rangle}^{\rho,\xi}}

\theoremstyle{plain}
\newtheorem{theorem}{Theorem}[section]
\newtheorem{lemma}[theorem]{Lemma}
\newtheorem{proposition}[theorem]{Proposition}
\newtheorem{corollary}[theorem]{Corollary}
\theoremstyle{definition}

\theoremstyle{remark}
\newtheorem{remark}[theorem]{Remark}

\begin{document}

\begin{abstract} 
We first prove an isomorphism between the moduli space of smooth cubic threefolds and the moduli space of hyperk\"ahler fourfolds of $K3^{[2]}$-type with a non-symplectic automorphism of order three, whose invariant lattice has rank one and is generated by a class of square~$6$; both these spaces are uniformized by the same 10-dimensional arithmetic complex ball quotient. We then study the degeneration of the automorphism along the loci of nodal or chordal degenerations of the cubic threefold, showing the birationality of these loci with some moduli spaces of hyperk\"ahler fourfolds of $K3^{[2]}$-type with non-symplectic automorphism of order three belonging to different families. Finally, we construct a cyclic Pfaffian cubic fourfold to give an explicit construction of a non-natural automorphism of order three on the Hilbert square of a K3 surface.
\end{abstract}

\maketitle

\section{Introduction}

The aim of this paper is to study certain moduli spaces of  $4$-dimensional irreducible holomorphic symplectic manifolds endowed with a non-symplectic group action of prime order, constructed in full generality by the authors in \cite{BCS_ball}, which are uniformized by the same $10$-dimensional arithmetic ball quotient as the smooth cubic threefolds \cite{ACT}, and to explore some deep consequences of this isomorphism.

Arithmetic complex ball quotients have attracted great interest in the last decades since they have been proven to be ubiquitous targets of period maps, allowing to show unexpected isomorphisms between different moduli spaces. For example, an arithmetic quotient of the complement of a hyperplane arrangement $\Delta$ in the $4$-dimensional complex ball  $\IB^4$ arising as period domain of smooth cubic surfaces in~\cite{ACT_surf} has been studied by Dolgachev--van Geemen--Kond\=o~\cite{DvGK}, who show that $\mathbb{B}^4\setminus \Delta$ and degenerations to $\Delta$ can also be interpreted  as moduli spaces of K3 surfaces endowed with certain non-symplectic automorphisms of order three.

Another famous example, which is the core of the present paper, is given by the quotient of the $10$-dimensional complex ball which contains as a Zariski open set the moduli space of smooth cubic threefolds, as shown by Allcock--Carlson--Toledo~\cite{ACT} and by Looijenga--Swierstra~\cite{LS}. We know nowadays that the hyperplane arrangements corresponding respectively to nodal and chordal degenerations of cubic threefolds are in fact birational to two other arithmetic complex ball quotients (see \cite{CMJL}), first studied by Kond\=o~\cite{Kondo} and appearing as the moduli spaces of K3 surfaces endowed with two different kinds of non-symplectic automorphisms of order three.

The main result of the first part of this paper is to establish an isomorphism between the moduli space~$\cC^\smooth_3$ of smooth cubic threefolds  and the moduli space~$\BCSmod$ of fourfolds deformation equivalent to the Hilbert square of a K3 surface endowed with a special non-symplectic automorphism of order three, following~\cite[Section~7.1]{BCS_ball}. We recall in Section~\ref{s_fano} the construction of Fano varieties of cyclic cubic fourfolds and of their automorphism, as done in~\cite[Example 6.4]{BCS_class}; then in Section~\ref{s_iso} we prove our first main result:

\begin{theorem}\label{th_main1}
The moduli spaces $\cC^\smooth_3$ and $\BCSmod$ are isomorphic.
\end{theorem}

This result will be proven by using the period map of irreducible holomorphic symplectic manifolds and Allcock--Carlson--Toledo's  period map for cubic threefolds. One of the consequences of Theorem~\ref{th_main1} is Corollary \ref{cor_aut}: although \apriori inside $\BCSmod$ we could have deformations which are non-isomorphic to Fano varieties of cyclic cubic fourfolds, this is not the case and every element in this moduli space is isomorphic to such a Fano variety of lines.

The rest of the paper is devoted to the exploration of more consequences of this isomorphism, in particular we give a geometric description of the degenerations of the non-symplectic automorphism of order three.
Indeed, the main part of ~\cite{ACT} is devoted to the extension of the period map to singular cubics, either nodal or chordal. In Section~\ref{s_degen} we will use this extension to understand geometrically the degenerations of the automorphism along the corresponding hyperplanes in the arithmetic ball quotient. The final outcome of this study are Propositions~\ref{prop_birat_nodal}~and~\ref{prop_birat_chordal}, which can be summarized as follows:

\begin{theorem}\label{th_main2}
The stable discriminant locus (corresponding to nodal degenerations), respectively the stable chordal locus (corresponding to stable chordal degenerations) are birational to $9$-dimensional moduli spaces of fourfolds which are deformations of a Hilbert square of a K3 surface endowed with a non-symplectic automorphism of order three, having respectively invariant lattice isometric to $U(3)\oplus\langle -2\rangle$ and to $U\oplus\langle -2\rangle$. 
\end{theorem}

Finally, in Section~\ref{s_pfaff}, looking at Hassett divisors on the moduli space of cubic fourfolds, we produce a Pfaffian cyclic cubic fourfold, thus we show:

\begin{corollary}\label{cor_main3}
There exists a smooth complex K3 surface whose Hilbert square admits a non-natural non-symplectic automorphism  of order three, with fixed locus isomorphic to the Fano surface of a cubic threefold
and invariant lattice isometric to~$\langle 6 \rangle$.
\end{corollary}

We further explicitly describe this automorphism in terms of the Pfaffian geometry of the K3 surface, in analogy with Beauville's construction of the non-natural non-symplectic involution on the Hilbert scheme of a general quartic surface~\cite{B_remarks}.

\subsection*{Acknowledgements}
The second author was partially supported by FIRB 2012 ``Moduli spaces and their applications" and by ``Laboratoire Internationale LIA LYSM".
The authors thank Klaus Hulek, Shigeyuki Kond\=o, Christian Lehn, Gregory Sankaran and Davide Veniani for their helpful comments and suggestions.

\section{Fano varieties of lines of cyclic cubic fourfolds}\label{s_fano}

We work over the field of complex numbers.

\subsection{Cyclic cubic fourfolds}

A general element of the linear system $\ls{\cO_{\proj 4} (3)}$ defines a smooth cubic threefold~$C$ in~$\proj 4$. The ramified triple covering $Y\to\proj 4$ branched along~$C$ is then a cyclic cubic fourfold in~$\proj 5$. 
We denote by $\sigma$ a generator of the fundamental group of the covering, 
it is a biregular automorphism of order three of~$Y$. 

By Hassett~\cite[Proposition~2.1.2]{Hassett} the lattice $\HH^4(Y,\IZ)$ endowed with the Poincar\'e pairing $\langle\alpha,\beta\rangle_Y\coloneqq\int_Y\alpha\beta$ is an odd unimodular 
lattice isometric to 
$$
\Lambda\coloneqq\langle 1\rangle^{\oplus 21}\oplus\langle -1\rangle^{\oplus 2},
$$
where $\langle m\rangle$
denotes the rank one lattice with quadratic form taking the value $m\in\IZ$ on the generator. 
Denote by $h\in\HH^2(\proj 5,\IZ)$ the class of a hyperplane section. 
The class $\theta_Y\coloneqq h^2_{|Y}$ in~$\HH^4(Y,\IZ)$ has square~$3$ and the primitive cohomology space $\HH^4_\prim(Y,\IZ)$ is by definition the orthogonal 
complement of $\theta_Y$ in $\HH^4(Y,\IZ)$. Then $\HH^4_\prim(Y,\IZ)$ is an even lattice of signature $(20,2)$, isometric to the lattice 
$$ 
S\coloneqq U^{\oplus 2}\oplus E_8^{\oplus 2}\oplus A_2,
$$ 
where $U$ is the hyperbolic plane and $E_8,A_2$  are the positive definite root lattices (see also~\cite[proof of Theorem~1.7]{ACT} or \cite[Section~1]{LS} for another point of view).
We fix for the rest of this paper a polarization class $\kappa\in\Lambda$ of square $3$ and a primitive embedding $S\hookrightarrow\Lambda$ such that $(\IZ\kappa)^\perp=S$.

Denote by $f(x_0,\ldots,x_4)=0$ an equation of $C$. Then $Y$ has equation 
$$
F(x_0,\ldots,x_5)\coloneqq f(x_0,\ldots,x_4)+x_5^3
$$
and the automorphism $\sigma$ maps $(x_0,\ldots,x_4,x_5)$ to $(x_0,\ldots,x_4,\xi x_5)$, where $\xi$ is a primitive third root of the unity that we fix for the rest of this paper. By Griffiths residue theorem, the Hodge group $\HH^{3,1}(Y)$ is one-dimensional, generated by $\Res\left(\frac{\omega_{\proj 5}}{F^2}\right)$, where 
$
\omega_{\proj 5}=\sum_{i=0}^5 (-1)^i x_i dx_0\wedge\ldots\wedge \widehat{dx_i}\wedge \ldots \wedge dx_5.
$
Since the residue map $\Res\colon\HH^5(\proj 5\setminus Y,\IC)\to\HH^4(Y,\IC)$ is $\sigma^\ast$-equivariant, we compute  $\sigma^\ast\Res\left(\frac{\omega_{\proj 5}}{F^2}\right)=\xi\Res\left(\frac{\omega_{\proj 5}}{F^2}\right)$. Denoting by $\HH^4_\prim(Y,\IC)_{\xi}$ the eigenspace  of $\sigma^\ast$ associated to the eigenvalue $\xi$, the line $\HH^{3,1}(Y)\subset \HH^4_\prim(Y,\IC)_{\xi}$ is called, following Allcock--Carlson--Toledo~\cite{ACT}, the \emph{period of the cubic threefold $C$}.


\subsection{Fano variety of lines on cyclic cubic fourfolds}\label{ss_fano}

Consider the Fano variety of lines $\Fano Y$, which parametrizes the projective lines contained in $Y$. 
By Beauville--Donagi~\cite{BD}, the variety $\Fano Y$ is irreducible holomorphic symplectic, 
deformation equivalent to the Hilbert square of a K3 surface. As a consequence, 
the second cohomology group with integer coefficients $\HH^2(\Fano Y,\IZ)$, endowed with the Beauville--Bogomolov--Fujiki bilinear form $\langle-,-\rangle_{\Fano Y}$ is isometric to the lattice 
$$
L\coloneqq U^{\oplus 3}\oplus E_8(-1)^{\oplus 2}\oplus\langle -2\rangle.
$$

The automorphism $\sigma$ of $Y$ leaves globally invariant the lines of $Y$ contained in the cubic $C$, so the fixed locus of 
the action of $\sigma$ induced on $\Fano Y$ is isomorphic to the Fano variety of lines~$\Fano C$, which is a surface of general type with Hodge numbers ${h^{1,0}=5}$, ${h^{2,0}=10}$ and ${h^{1,1}=25}$ (see~\cite[Example~6.4]{BCS_class},\cite[Formula~(0.7)]{CG}). It follows from the lattice-theoretical classification of non-symplectic automorphisms in~\cite[Example~6.4]{BCS_class} that the invariant lattice $\HH^2(\Fano Y,\IZ)^\sigma$ is isometric to $\langle 6 \rangle$, with orthogonal complement in $\HH^2(\Fano Y,\IZ)$, denoted by $\HH^2_\orth(\Fano Y,\IZ)$, isometric to the lattice
$S(-1)= U^{\oplus 2}\oplus E_8(-1)^{\oplus 2}\oplus A_2(-1)$.

The apparent coincidence with the results of the preceding section is due to the Abel--Jacobi map.
Denote by $Z\subset \Fano Y\times Y$ the universal family and by 
$p$, \resp~$q$, the projection to $\Fano Y$, \resp~$Y$.
By \cite[Proposition 4]{BD} the 
Abel--Jacobi map 
$$
A\coloneqq p_\ast q^\ast\colon \HH^4(Y,\IZ)\lra \HH^2(\Fano Y,\IZ)
$$ 
is an isomorphism of Hodge 
structures. Denote by 
$\theta_{\Fano Y}\in\HH^2(\Fano Y,\IZ)$ the class of a hyperplane section in the Pl\"ucker embedding. It is easy to see geometrically that
$A(\theta_Y)=\theta_{\Fano Y}$. Since $Y/\langle\sigma\rangle\cong\proj 4$, we can directly deduce that the invariant lattice $\HH^4(Y,\IZ)^\sigma$ has rank one, hence it is generated by
$\theta_Y$. By $\sigma$-equivariance of $A$ we observe that
the invariant lattice $\HH^2(\Fano Y,\IZ)^\sigma$ is  generated by $\theta_{\Fano Y}$. 

Moreover, by \cite[Proposition~6]{BD} we have $A(\HH^4_\prim(Y,\IZ))=\HH^2_\orth(\Fano Y,\IZ)$ and
$$
\langle A(\alpha),A(\beta)\rangle_{\Fano Y}=-\langle\alpha,\beta\rangle_Y\quad \forall \alpha,\beta\in\HH^4_\prim(Y,\IZ).
$$

Since  $A(\HH^{3,1}(Y))= \HH^{2,0}(\Fano Y)$, by $\sigma$-equivariance
of $A$ we deduce from the results of the previous section that 
the line $\HH^{2,0}(Y)$, which is the \emph{period of the irreducible holomorphic symplectic manifold} $\Fano Y$, lives in the eigenspace $\HH^2_\orth(\Fano Y,\IC)_{\xi}$ of $\HH^2_\orth(\Fano Y,\IC)$ associated to the eigenvalue $\xi$. The transition from the lattice $S$ to the opposite lattice $S(-1)$ is explained by the Abel--Jacobi map $A$ being an anti-isometry at the level of the primitive cohomology.


\section{Occult period maps and unexpected isomorphisms}\label{s_iso}

\subsection{Two useful lemmas}

We state two classical results of Nikulin~\cite{Nikulin} on lattice theory which will be used several times below. For any integral lattice $S$, we denote by $S^\ast$ the dual lattice and by $D_S\coloneqq S^\ast/S$ its discriminant group, 
endowed with its quadratic form $q_S$. The \emph{length} $\ell(S)$ of $D_S$ is its minimal number of generators.

\begin{lemma}\label{lem1}\cite[Theorem~1.14.2]{Nikulin} Assume that $S$ is an even indefinite lattice. If $\rk(S)\geq \ell(S)+2$, then the natural map $\Orth(S)\to\Orth(D_S)$ is surjective.

\end{lemma}

\begin{lemma}\label{lem2}
Let $M$ be an integral non-degenerate lattice, $S$ a primitive sublattice of~$M$ and $T\coloneqq S^\perp$ its orthogonal complement. Let $\phi_S\in\Orth(S)$ and $\phi_T\in\Orth(T)$. The isometry $\phi\coloneqq(\phi_S,\phi_T)$ of $S\oplus T$ extends to an isometry of $M$ if and only if its action on $D_S\oplus D_T$ leaves globally invariant the subgroup $\frac{M}{S\oplus T}$. In particular, if $\phi_S$ acts by $\epsilon\,\id$ on $D_S$, with $\epsilon\in\{-1,1\}$, then $(\phi_S,\epsilon\,\id_T)$ extends to an isometry of $M$.
\end{lemma}

\begin{proof}
The inclusions $S\oplus T\subset M\subset M^\ast\subset S^\ast\oplus T^\ast$ produce a subgroup $\frac{M}{S\oplus T}\subset D_S\oplus D_T$. For any $m\in M$, one has $\lambda m=s+t\in S\oplus T$ for some $\lambda\in\IZ$. Consider $m=\frac{s+t}{\lambda}\in\frac{M}{S\oplus T}$. If this group is $\phi$-stable in $D_S\oplus D_T$, then $\phi(m)=\frac{\phi_S(s)+\phi_T(t)}{\lambda}\in\frac{M}{S\oplus T}$, so $\phi(m)\in M$ and $\phi$ extends to an isometry of $M$.
\end{proof}

\subsection{Moduli space of smooth cubic threefolds}

We denote by $\cC^\stable_n$ the GIT moduli space of $\PGL_{n+2}(\IC)$-stable points in the linear system $\ls{\cO_{\proj {n+1}}(3)}$. Inside this quasi-projective variety,
the locus $\cC^{\smooth}_n$ whose points parametrize projective equivalence classes of smooth cubic $n$-folds in $\proj {n+1}$ is the open set determined by the nonvanishing of the discriminant (see~\cite[Chapter~5]{Mukai}). 

By Allock--Carlson--Toledo~\cite{ACT}, the moduli space $\cC^\smooth_3$ is isomorphic to the quotient by an arithmetic group of the complementary of a hyperplane arrangement in a $10$-dimensional complex ball. 
We briefly recall the main ingredients for later use. Given $C\in\cC^\smooth_3$ (we choose a representative of the projective equivalence class) we denote by $Y\to\IP^4$ the ramified cyclic triple covering branched along $C$ and by~$\sigma$ the generator of the covering group which acts by multiplication by $\xi$ on the period $\HH^{3,1}(Y)$. Since $\HH^4_\prim(Y,\IZ)$ has no fixed points for the action of $\sigma$, 
it inherits a natural structure of free module over the ring of Eisenstein integers $\cE\coloneqq\IZ[\xi]$, which does not depend on the cubic $C$. Hence any \emph{marking} (\ie isometry of $\IZ$-lattices) $\eta\colon\HH^4_\prim(Y,\IZ)\to S$ induces a representation 
$$
\IZ/3\IZ\to\Orth(S), \quad 1\mapsto \eta\circ\sigma^\ast\circ\eta^{-1}\eqqcolon \rho_0
$$ 
whose isomorphism class does not depend on $C$ by the local topological triviality of the family of cubic threefolds. We call \emph{framing} a marking compatible with the representation, in the sense that
$\eta\circ\sigma^\ast=\rho_0\circ\eta$. Said differently, a framing $\eta\colon \HH^4_\prim(Y,\IZ)\to S$ is an isometry of $\IZ$-lattices and an isomorphism of $\cE$-modules. 

We denote by  $\cF^\smooth_{3}$ the moduli space of framed smooth cubic threefolds: it parametrizes pairs $(C,\eta)$, where $C\in\cC^\smooth_3$ and $\eta\colon \HH^4_\prim(Y,\IZ)\to S$ is a framing considered up to composition on the target with the group $\mu_6\coloneqq\{\pm 1,\pm\rho_0,\pm\rho_0^2\}$ of $6$th roots of the unity, which is also  the group of units of $\cE$. We denote by $\Gamma_S^{\rho_0,\xi}$ the group of isometries~$\gamma$ of~$S$ such that $\gamma\circ\rho_0=\rho_0\circ\gamma$. The group $\Gamma_S^{\rho_0,\xi}$ is equivalently the group of isometries of the $\cE$-valued Hermitian form on $S$ naturally associated to~$\rho_0$ (see Section~\ref{ss_hermitian}). It acts on~$\cF^\smooth_3$ on the left by
$(\gamma,(C,\eta))\mapsto (C,\gamma\circ\eta))$ for all $\gamma\in\Gamma_S^{\rho_0,\xi}$,
and the units act trivially so we can consider the action of the projective group $\IP\Gamma_S^{\rho_0,\xi}\coloneqq\Gamma_S^{\rho_0,\xi}/\mu_6$. We  thus have a natural isomorphism $\cC^\smooth_3\cong\leftquo{\IP\Gamma_S^{\rho_0,\xi}}{\cF^\smooth_{3}}$. 

Any framing $\eta\colon \HH^4_\prim(Y,\IZ)\to S$ induces an isomorphism $\eta\colon\HH^4_\prim(Y,\IC)_{\xi}\to S_{\xi}$, where~$S_{\xi}$ is the eigenspace of $S\otimes_\IZ\IC$ for the eigenvalue $\xi$ of the isometry $\rho_0$. The Poincar\'e pairing on $\HH^4(Y,\IZ)$ extends to a hermitian form ${h_Y(\alpha,\beta)\coloneqq\langle\alpha,\bar\beta\rangle_Y}$ on $\HH^4(Y,\IC)$. Using the Hodge--Riemann relations~\cite[p.123]{GH} we get that $\HH^4(Y,\IC)_{\xi}$ has signature $(10,1)$ and that the line $\HH^{3,1}(Y)$ is negative-definite.
Denoting by $h_S(\alpha,\beta)\coloneqq\langle \alpha,\bar\beta\rangle_S$ the hermitian form on $S\otimes_\IZ\IC$ induced by the integral bilinear form $\langle-,-\rangle_S$ on $S$, we see that the line $\eta(\HH^{3,1}(Y))$ defines a point in the $10$-dimensional complex ball
$$
\Omega_S^{\rho_0,\xi}\coloneqq\left\{\omega\in\IP\left(S_{\xi}\right)\,|\,h_S(\omega,\omega)<0\right\}\cong \IB^{10}.
$$
The period map $\cF^\smooth_{3}\to\Omega_S^{\rho_0,\xi}$, $(C,\eta)\mapsto \eta(\HH^{3,1}(Y))$ commutes with the action of $\IP\Gamma_S^{\rho_0,\xi}$ and defines a quotient period map
$$
\cP_S^{\rho_0,\xi}\colon \cC^\smooth_3\lra \frac{\IB^{10}}{\IP\Gamma_S^{\rho_0,\xi}}.
$$
By~\cite[Theorem~1.9]{ACT} the period map $\cP_S^{\rho_0,\xi}$ is an isomorphism onto its image, which is the complementary of the quotient of a hyperplane arrangement (see Section~\ref{s_degen}).

\subsection{Relation with the period map for cubic fourfolds}
\label{ss_compare}
As initial data, we choose a cubic threefold $\tilde C\in\cC^\smooth_3$, its associated cyclic fourfold ${\tilde Y\in\cC^\smooth_4}$ and a framing $\tilde\eta\colon\HH^4_\prim(\tilde Y,\IZ)\to S$.
The period of $\tilde Y$ is the line $\HH^{3,1}(\tilde Y)$ inside $\HH^4_\prim(\tilde Y,\IC)$. The period domain
$\Omega_S\coloneqq\left\{\omega\in\IP\left(S_\IC\right)\,|\langle\omega,\omega\rangle_S=0,h_S(\omega,\omega)<0\right\}$
is isomorphic to the Grassmannian variety of negative definite $2$-planes in $S$, and has two connected components interchanged by complex conjugation. We denote by $\Omega_S^\circ$ the connected component which contains $\tilde\eta(\HH^{3,1}(\tilde Y))$.
By deformation of the complex structure we get a holomorphic period map from the moduli space of \emph{marked} cubic fourfolds $\cC^\smooth_{4,\marked}\to \Omega_S^\circ$.
By results of Ebeling and Beauville (see \cite[Theorem~2]{B_monodromy}), the monodromy group of a cubic fourfold is isometric to the group
$\Gamma_S$ of isometries of $S$ which act trivially on the discriminant group $D_S$ and which respect the orientation of a negative-definite $2$-plane in $S$ (see Laza~\cite[\S 2.2]{Laza_period}). Thus $\Gamma_S$ acts on $\Omega_S^\circ$ and by a theorem of Voisin~\cite{VoisinCubic}, the period map $\cP_S\colon\cC^\smooth_4\to \frac{\Omega_S^\circ}{\Gamma_S}$ is an open embedding. 
Note that $\Gamma_S^{\rho_0,\xi}$ is not a subgroup of $\Gamma_S$: since $D_S\cong\IZ/3\IZ$, the elements of 
$\Gamma_S^{\rho_0,\xi}$ act by $\pm 1$ on $D_S$ (see Lemma~\ref{lem1}). Observe that $\rho_0\in\Gamma_S$: the isometry~$\rho_0$ comes from the covering automorphism of $Y$ so it fixes the polarization~$\theta_Y$ and it acts trivially on the discriminant group $D_S\cong \IZ\theta_Y/3\IZ\theta_Y$ since the lattice $\HH^4(Y,\IZ)$ is unimodular. Thus the $\IP\Gamma_S^{\rho_0,\xi}$-orbits in $\Omega_S$ are contained in the corresponding $\Gamma_S$-orbits and we have a commutative diagram:
$$
\xymatrix{\cC^\smooth_3\ar@{^(->}[r]^{\cP_S^{\rho_0,\xi}}\ar[d]_{\cov} & \frac{\Omega_S^{\rho_0,\xi}}{\IP\Gamma_S^{\rho_0,\xi}}\ar[d]\\
\cC^\smooth_4\ar@{^(->}[r]^{\cP_S} & \frac{\Omega_S^\circ}{\Gamma_S}}
$$
It is well-known that the \emph{covering morphism} $\cov$ mapping $C$ to $Y$ is generically injective (see Beauville~\cite[proof of Theorem~4.6]{B_modcub} and references therein for a similar argument).

\subsection{The Hermitian module}
\label{ss_hermitian}

The ring of Eisenstein integers $\cE\coloneqq\IZ[\xi]$ is euclidean with invertible group $\cE^\ast=\{\pm 1,\pm\xi,\pm\xi^2\}$. Its irreducible elements are the prime numbers $p\in\IZ$ such that $p\equiv 2\mod(3)$ and the elements $x\in\cE$ whose norm $x\bar{x}$ is a prime integer. In particular $\theta\coloneqq \xi-\xi^{-1}$ is irreducible in~$\cE$ and $3=\theta\bar{\theta}$.
Recall that $\cE$ acts on $S$ by $\xi x\coloneqq\sigma(x)$ for all $x\in S$.
Following~\cite[Chapter~1]{ACT}, we define a Hermitian form on $S$ by
$$
H_S(x,y)\coloneqq \theta\left(\langle x,\xi y\rangle_S-\xi\langle x,y\rangle_S\right).
$$
It is an easy exercise to check that $H_S$ is a Hermitian form and that it takes values in the prime ideal $\theta\cE$ of $\cE$. One has $\langle \delta,\delta\rangle_S=2$
if and only if $H_S(\delta,\delta)=3$: this characterizes the \emph{roots} of $S$. For such a root $\delta$, the image of the $\cE$-linear morphism $H_S(-,\delta)\colon S\to \cE$ is an ideal generated by an element $\alpha_\delta$ such that $3\cE\subset\alpha_\delta\cE\subset \theta\cE$. It follows easily that either $H_S(S,\delta)=\theta\cE$ ($\delta$~is called \emph{nodal}) or $H_S(S,\delta)=3\cE$ ($\delta$~is called \emph{chordal}).


\subsection{A moduli space of non-symplectic automorphisms}\label{ss_modnonsympl}

We start again from a ramified cyclic covering $Y\to\proj 4$ branched along the cubic $C$, with covering automorphism~$\sigma$ acting by multiplication by $\xi$. Any marking $\eta\colon\HH^4_\prim(Y,\IZ)\to S$ induces through the
 Abel--Jacobi map $A$ an isometry $\eta\circ A^{-1}\colon\HH^2_\orth(\Fano Y,\IZ)\to S(-1)$, which extends to a marking $\HH^2(\Fano Y,\IZ)\to L$ as follows. By~\cite[Theorem~3.7]{BCS_class} 
the lattice $S(-1)$ admits a unique primitive embedding in $L$, up to an isometry of $L$, and its orthogonal complement $S(-1)^\perp$ is isometric to the lattice~$\langle 6\rangle$. We implicitly fix this embedding.
 This determines, up to an isometry of~$L$, a primitive embedding $j\colon \langle 6\rangle\hookrightarrow L$ that we fix for the rest of this paper. 
We extend~$\eta$ to an isometry $\HH^2_\orth(\Fano Y,\IZ)\oplus\IZ\theta_{\Fano Y}\to S(-1)\oplus j(\langle 6\rangle)$ which extends using Lemma~\ref{lem2} to an isometry $\tilde\eta\colon \HH^2(\Fano Y,\IZ)\to L$. The action of the automorphism $\sigma$ defines as above a representation 
$$
\IZ/3\IZ\lra\Orth(L), \quad 1\mapsto \tilde\eta\circ\sigma^\ast\circ\tilde\eta^{-1}\eqqcolon\rho
$$
whose isomorphism class does not depend on the cubic $C$, and whose invariant subrepresentation is $j(\langle 6\rangle)$. The restriction of $\rho$ to $S(-1)$ is nothing else than $\rho_0$ viewed as an isometry of $S(-1)$ instead of $S$.
Following Boissi\`ere--Camere--Sarti~\cite{BCS_ball}, the $4$-tuple $(\Fano Y,\tilde\eta,\sigma,\iota)$ is the initial data of the construction of a \emph{moduli space of $(\rho,j)$-polarized 
irreducible holomorphic symplectic manifolds}. We briefly recall the main ingredients for later use.

Denote by $\cM_{L}$ the moduli space of equivalence classes of pairs $(X,\eta)$ where $X$  is an irreducible holomorphic symplectic manifold 
deformation equivalent to the Hilbert square of a K3 surface and $\eta\colon\HH^2(X,\IZ)\to L$ is a marking (\ie an isometry of $\IZ$-lattices), 
where two pairs $(X,\eta)$ and $(X',\eta')$ are considered equivalent if there exists a biregular isomorphism $f\colon X\to X'$ such that $\eta=\eta'\circ f^\ast$.
 Following~\cite[Section~2]{Markman_Torelli} and ~\cite{BCS_ball} we fix once and for all a connected component $\cM_{L}^\circ$ of $\cM_{L}$ and we define
a $(\rho,j)$-\emph{polarisation} of an irreducible holomorphic symplectic manifold~$X$
deformation equivalent to the Hilbert square of a K3 surface as the data of:
\begin{itemize} 
\item a marking~$\eta$,
\item a primitive embedding $\iota\colon \langle 6\rangle\hookrightarrow\NS(X)$ such that $\eta\circ\iota=j$,
\item an automorphism $\sigma\in\Aut(X)$ such that $\eta$~is a framing for $\sigma\in\Aut(X)$ (\ie $\eta\circ\sigma^\ast=\rho\circ \eta$) and $\sigma_{|\HH^{2,0}(X)}=\xi\id$.
\end{itemize} 
If such an automorphism $\sigma$ exists, then it is uniquely determined by this property (see \cite[Lemme~1.2]{Mongardi_CRAS}).
We denote by $\BCSmodfr$ (the embedding $j$ is 
implicit in this notation) the set of equivalence classes of such tuples $(X,\eta,\sigma,\iota)$, where two tuples $(X,\eta,\sigma,\iota)$ and $(X',\eta',\sigma',\iota')$ are equivalent  if there exists a biregular isomorphism $f\colon X\to X'$ such that $\eta'=f^\ast\circ\eta$, $\sigma'=f\circ\sigma\circ f^{-1}$ and $\iota'=f^\ast\circ \iota$. 
We similarly define the space $\cM_{\langle 6 \rangle}$ parametrizing $j$-polarized triples $(X,\eta,\iota)$ up to equivalence.
We have then the inclusions $\BCSmodfr\subset  \cM_{\langle 6 \rangle}\subset \cM^\circ_{L}$.

\begin{remark} \label{rem_ample}
The polarisation $\iota(\langle 6\rangle)$ contains always an ample class: the variety~$X$ is projective since it admits a non-symplectic automorphism 
(see~\cite[\S 4]{B_remarks}), so if $\ell\in\NS(X)$ is an ample class, the invariant ample class $\ell+\sigma^\ast \ell+(\sigma^\ast)^2 \ell$ 
is necessarily a multiple of the generator of the rank one invariant lattice. Denote $\vartheta\in\iota(\langle 6\rangle)$ the primitive ample invariant class;
since the invariant lattice has rank one, this class is uniquely defined by the automorphism $\sigma$. It follows that
the map $(X,\sigma,\eta,\iota)\mapsto (X,\eta,\vartheta)$ embeds $\BCSmodfr$ in the moduli space of irreducible holomorphic symplectic manifolds polarized by an ample primitive
class of square $6$ studied by Gritsenko--Hulek--Sankaran~\cite{GHS_moduli}.
\end{remark}

Starting from the initial data $(\Fano Y,\tilde\eta,\sigma,\iota)$ as above, we consider the monodromy group $\Mon^2(\Fano Y)\subset \Orth(\HH^2(\Fano Y,\IZ))$ and we define 
$$
\Mon^2(L)\coloneqq \tilde\eta\circ\Mon^2(\Fano Y)\circ\tilde\eta^{-1}\subset\Orth(L).
$$ 
It is easy to check that this group does not depend on the representative of the equivalence class in $\BCSmod$ and that it is conjugated to the monodromy group of any variety deformation equivalent to $\Fano Y$. 
Following~\cite[Section~6]{BCS_ball}, we define the group $\Gamma_{\langle 6\rangle}^{\rho,\xi}$ as the image in $\Orth(S(-1))$ of those elements $\gamma\in\Mon^2(L)$ which act trivially on $j(\langle 6\rangle)$ and commute with $\rho$.
The group $\Gamma_{\langle 6\rangle}^{\rho,\xi}$ acts on  $\BCSmodfr$ by $(\gamma,(X,\eta,\sigma,\iota))\mapsto(X,\gamma\circ\eta,\sigma,\iota)$ and we define $\BCSmod\coloneqq\leftquo{\Gamma_{\langle 6\rangle}^{\rho,\xi}}{\BCSmodfr}$.

Using the properties of the Beauville--Bogomolov--Fujiki quadratic form, we see that the line $\eta(\HH^{2,0}(X))$ defines a point in the $10$-dimensional complex ball
$$
\Omega_{\langle 6\rangle}^{\rho,\xi}\coloneqq\{\omega\in\IP(S(-1)_{\xi})\,|\, h_{S(-1)}(\omega,\omega)>0\}\cong\IB^{10}.
$$
The period map $\BCSmodfr\to\Omega_{\langle 6\rangle}^{\rho,\xi}$, $(X,\eta,\sigma,\iota)\mapsto\eta(\HH^{2,0}(X))$ commutes with the action of $\Gamma_{\langle 6\rangle}^{\rho,\xi}$ and defines a quotient period map
$$
\cP_{\langle 6\rangle}^{\rho,\xi}\colon\BCSmod\lra\frac{\IB^{10}}{\Gamma_{\langle 6\rangle}^{\rho,\xi}}
$$
By~\cite[Theorem~4.5, Theorem~5.6, Section~7.1]{BCS_ball}, the moduli space $\BCSmod$ is connected, 
the period map is an open embedding and its image is the complementary of the quotient of a hyperplane arrangement.
Note that there is a tiny difference between this definition of the moduli space and those of ~\cite{BCS_ball}
since we include $\sigma$ and $\iota$ in the data. Since they are uniquely determined, this is of course equivalent
and our period map is, strictly speaking, the composition of a bijective forgetful map and of the period map of~\cite{BCS_ball}. 
We define the structure of complex manifold of $\BCSmod$ as the one inherited by the one of its period domain.

\begin{theorem}\label{th_isom_mod_spaces}
The moduli spaces $\cC^\smooth_3$ and $\BCSmod$ are isomorphic.
\end{theorem}

\begin{proof}
The period domains $\Omega_S^{\rho_0,\xi}$ and $\Omega_{\langle 6\rangle}^{\rho,\xi}$ are clearly identical, we identify them with the complex ball $\IB^{10}$. 
By~\cite[Theorem~6.1]{ACT} the image of the period map $\cP_C$  is the complementary of the quotient of the union of the hyperplanes $\delta^\perp\cap\IB^{10}$ for all  \emph{roots} $\delta$ of $S$, \ie all elements of norm $\langle \delta,\delta\rangle_S=2$ (in~\cite{ACT} these have norm $3$ for the naturally associated $\cE$-valued Hermitian form on $S$, see Section~\ref{ss_hermitian}). 
By~\cite[Theorem~4.5, Theorem~5.6, Section~7.1]{BCS_ball}, the image of the period map $\cP_{\langle 6\rangle}^{\rho,\xi}$  is the complementary of the quotient of the union of the hyperplanes $\delta^\perp\cap\IB^{10}$ for all $\delta\in S(-1)$ which are \emph{monodromy birationally minimal (MBM)}, see~\cite{AV}. By~\cite{BHT,Mongardi} those are exactly the classes either of norm $-2$ or of norm $-10$ and divisibility $2$ (\ie $\langle S,\delta\rangle_S=2\IZ$). The latter cannot happen since $S(-1)$ has discriminant $3$. This shows that both hyperplane arrangements are identical.

By a result of Markman~\cite[Theorem~1.2]{Markman_constraints},\cite[\S 9]{Markman_Torelli} $\Mon^2(\Fano Y)$ is the subgroup of $\Orth(\HH^2(\Fano Y,\IZ))$ of isometries leaving globally invariant the positive cone of $\Fano Y$.
By Remark~\ref{rem_ample}, isometries acting trivially on the polarization $\iota(\langle 6\rangle)$ do automatically fix an ample class, so they leave globally invariant the positive cone. The elements of $\Gamma_{\langle 6\rangle}^{\rho,\xi}$ are thus the restrictions to $S(-1)$ of those isometries of $L$ acting trivially on $j(\langle 6\rangle)$ and commuting with the representation $\rho$.  By Lemma~\ref{lem2}, any isometry of $S(-1)$ commuting with $\rho_0$ extends to an isometry of $L$ acting by $\pm 1$ on $j(\langle 6\rangle)$ and commuting with $\rho$.  So we have an isomorphism $\Gamma_{\langle 6\rangle}^{\rho,\xi}\cong \frac{\Gamma_S^{\rho_0,\xi}}{\langle\pm \id_S\rangle}$ and the ball quotients are equal.
\end{proof}

Concretely, we have a natural map $\rF\colon \cC^\smooth_3\to\BCSmod$ which maps the equivalence class of a framed cubic $(C,\eta)$ to the equivalence class of the $(\rho,j)$-polarized irreducible holomorphic manifold $(\Fano Y,\tilde\eta,\sigma,\iota)$ constructed above, and the theorem says that this map is an isomorphism. We thus have a commutative diagram
$$
\xymatrix{\cC^\smooth_3\ar[rr]^F\ar[dr]_{\cP_S^{\rho_0,\xi}}^\sim && \BCSmod\ar[dl]_\sim^{\cP_{\langle 6\rangle}^{\rho,\xi}}\\
& \displaystyle\frac{\IB^{10}\setminus\calH}{\IP\Gamma}}
$$
where $\Gamma\coloneqq \Gamma_{\langle 6\rangle}^{\rho,\xi}\cong \frac{\Gamma_S^{\rho_0,\xi}}{\langle\pm \id_S\rangle}$ and $\calH=\calH_\nodal\cup \calH_\chordal$ is the hyperplane arrangement discussed in the proof of Theorem~\ref{th_isom_mod_spaces}, where we denote
\begin{align*}
\calH_\nodal&\coloneqq \bigcup\limits_{\substack{\delta \text{ nodal root}\\\text{of } S}} \delta^\perp\cap \Omega_S^{\rho_0,\xi},\\
\calH_\chordal&\coloneqq \bigcup\limits_{\substack{\delta \text{ chordal root}\\\text{of } S}} \delta^\perp\cap \Omega_S^{\rho_0,\xi}
\end{align*}
the unions of hyperplanes orthogonal respectively to nodal and chordal roots.

\begin{corollary}
An irreducible holomorphic symplectic manifold $X$ deformation equivalent to the Hilbert square of a K3 surface, polarized by an ample class $\vartheta$ of square $6$, admits a non-symplectic automorphism of order three with invariant lattice $\IZ\vartheta\cong\langle 6\rangle$ if and only if $X$ is isomorphic to the Fano variety of lines of the ramified cyclic triple covering $Y$ of~$\proj 4$ branched along a smooth cubic threefold.
In this case, the automorphism $\sigma$ is induced by the covering automorphism.
\end{corollary}

\begin{remark}
Take $(X,\eta,\sigma,\iota)\in\BCSmodfr$. A class $d\in\HH^2(X,\IZ)$ is algebraic if and only if it is orthogonal to $\HH^{2,0}(X)$, or equivalently if the period $\eta(\HH^{2,0}(X))$ is outside the hyperplane orthogonal to $d$ in the period domain $\Omega_{\langle 6\rangle}^{\rho,\xi}$. Since such hyperplanes form a countable union, their complementary is dense in the period domain, so generically the N\'eron--Severi group of $X$ is generated by $\iota(\langle 6\rangle)$. The moduli space $\BCSmod$ is thus a $10$-dimensional moduli space of irreducible holomorphic symplectic manifolds deformation equivalent to the Hilbert square of a K3 surface, admitting a non symplectic automorphism of order three, and whose general Picard number is one. This answers a question asked to the authors by Olivier Debarre~\cite{DM}.
\end{remark}

\subsection{Relation with the moduli space of genus five principally polarized abelian varieties}

The term ``occult'' in the section title is a reference to a paper of Kudla--Rapoport~\cite{KudlaRapoport}. As suggested to us by Michael Rapoport and thanks to enlightening discussions with Gregory Sankaran, we can draw direct paths between all the involved period maps, producing in particular a period map from~$\BCSmod$ to a moduli space of principally polarized abelian varieties.
The classical period map for cubic threefolds of Clemens--Griffiths~\cite{CG} maps a cubic $C\in\cC^\smooth_3$ to its intermediate Jacobian $\IntJac(C)\coloneqq \frac{\HH^{1,2}(C)}{\HH^3(C,\IZ)}$, which gives a point in the $15$-dimensional moduli space $\cA_5$ of principally polarized abelian variety of genus $5$. 

\begin{corollary}
There exists an embedding $\BCSmod\hookrightarrow\cA_5$ mapping a $(\rho,j)$-polarized tuple $(X,\eta,\sigma,\iota)$ to the Albanese variety $\Alb(X^\sigma)$ of the fixed locus $X^\sigma\subset X$ of the automorphism $\sigma$.
\end{corollary}

\begin{proof}
By~\cite[(0.8)]{CG} the intermediate Jacobian of $C$ is isomorphic to the Albanese variety of the Fano variety of lines on $C$, \ie $\IntJac(C)\cong\Alb(\Fano C)$. As observed in Section~\ref{ss_fano}, the variety $\Fano C$ is the fixed locus in $\Fano Y$ of the automorphism $\sigma$ of the ramified covering map $Y\to\proj 4$ branched along $C$. Since the morphism ${\rF\colon\cC^\smooth_3\to\BCSmod}$ mapping $C$ to $(F(V),\sigma)$ (with the appropriate markings and polarizations as above) is an isomorphism, we get the result.
\end{proof}


\section{Degenerations}\label{s_degen}

Following an idea used by Dolgachev--Kond\=o~\cite[\S11]{DK} to study the degeneration of K3 surfaces with a non-symplectic automorphism of prime order,
we study the degenerations of those irreducible holomorphic symplectic manifolds $(X,\eta,\sigma,\iota)$ parametrized by $\BCSmodfr$, inside the moduli space $\cM_{\langle 6\rangle}$ of $j$-polarized marked triples.
Denote by $\Omega_{\langle 6\rangle}^\circ$ the connected component of $\{\omega\in\IP(S_\IC)\,|\,h_{S(-1)}(\omega,\omega)>0\}$ which contains the ball $\Omega_{\langle 6\rangle}^{\rho,\xi}$, it is easy to deduce from the theorem of surjectivity of the period map of Huybrechts~\cite[Theorem~8.1]{Huybrechts_Invent} that the period map
$\cP_{\langle 6\rangle}\colon \cM_{\langle 6\rangle}\to\Omega_{\langle 6\rangle}^\circ$
is surjective. Note that, comparing with the notation in Section~\ref{ss_compare}, we have $\Omega_{\langle 6\rangle}^\circ=\Omega_S^\circ$ but we use both notations depending on the context. We thus have a commutative diagram
$$
\xymatrix{
\BCSmodfr \ar@{^(->}[d]\ar[r]^{\cP_{\langle 6\rangle}^{\rho,\xi}}& \Omega_{\langle 6\rangle}^{\rho,\xi}\setminus\calH\ar@{^(->}[d]\\
\cM_{\langle 6\rangle}\ar@{->>}[r]^{\cP_{\langle 6 \rangle}} & \Omega_{\langle 6 \rangle}^\circ}
$$
Let $\omega\in\calH$ and $(X,\eta,\iota)\in\cP_{\langle 6\rangle}^{-1}(\omega)$. By~\cite[Theorem~1.2]{Markman_Torelli}, any two points in this fiber of $\cP_{\langle 6\rangle}$ correspond to birational manifolds.

We first observe that the isometry $\rho\in\Orth(L)$ is not 
represented by any automorphism of $X$. As explained in the proof of Theorem~\ref{th_isom_mod_spaces}, this is a special case of~\cite[Theorem~4.0.8]{BCS_class}:
 if $\omega\in\calH$, then $\omega$ is orthogonal to some class $\delta\in S$ such that $\eta^{-1}(\delta)$ is a MBM class of $X$. 
If $\rho$ was represented by an automorphism of $X$, this one would be automatically non-symplectic since~$\omega$ is not in the invariant sublattice of~$\rho$, so as explained in Remark~\ref{rem_ample} the line~$\iota(\langle 6\rangle)$ contains an ample class~$\vartheta$. But  since~$\delta\in S$, the divisor~$\eta^{-1}(\delta)$ is orthogonal to~$\vartheta$.
As explained in ~\cite[Remark~4.0.7 (2)]{BCS_class}, this is not possible since MBM classes have nonzero intersection with the K\"ahler cone of~$X$.
In particular, in this situation the line~$\iota(\langle 6\rangle)$ contains no ample class any more.

Following~\cite[Section~1]{Markman_Torelli}, we denote by $\Mon^2_\Hdg(X)$ the subgroup of $\Mon^2(X)$ consisting of those monodromies whose $\IC$-linear extension preserves the Hodge decomposition.
For any nonisotropic vector $v\in L$, we denote by $r_v\coloneqq x\mapsto x-2\frac{\langle x,v\rangle}{\langle v,v\rangle} v$ the reflection which leaves $v^\perp$ invariant. 
We consider the normal subgroup $W_\Exc(X)$ of $\Mon^2_\Hdg(X)$ generated by the reflections by prime exceptional divisors of $X$ (\ie  reduced and irreducible effective divisors of negative Beauville--Bogolomov--Fujiki degree).

\begin{proposition}\label{prop_birational_map}
Let $\omega\in\calH$ and $(X,\eta,\iota)\in\cP_{\langle 6\rangle}^{-1}(\omega)$. There exists a unique birational map $\beta\in\Bir(X)$ and a unique element
 $w\in W_\Exc(X)$ such that:
$$
\eta^{-1}\circ\rho\circ\eta=w\circ\beta^\ast.
$$
\end{proposition}

\begin{proof}
Since $\rho\in\Mon^2(L)$, the isometry $\psi\coloneqq \eta^{-1}\circ\rho\circ\eta$ is a monodromy operator of $X$. Since $\omega\in\IP(S(-1)_\xi)$, we have $\psi(
\HH^{2,0}(X))=\HH^{2,0}(X)$ so $\psi\in\Mon^2_\Hdg(X)$.
By~\cite[Theorem~1.6]{Markman_Torelli}, there exists a unique element  $w\in W_\Exc(X)$ and a birational map $\beta\in\Bir(X)$ such that $
\eta^{-1}\circ\rho\circ\eta=w\circ\beta^\ast$. The map $\beta$ is unique in this case since the natural map $\Aut(X)\to\Orth(\HH^2(X,\IZ))$ 
is injective in this deformation class
(see~\cite[Lemma~1.2]{Mongardi_CRAS}).
\end{proof}

This results means that the isometry $\rho$, after composition by a suitable product of reflections, is realized by a birational transformation of $X$. 
By a deeper analysis of the degeneracy situations, we will show below that the product of reflections $w$ is not trivial and that $\rho$ itself is in fact realized by another automorphism and not only by a birational map. This means that, when the period point $\omega\in\Omega_{\langle 6\rangle}^{\rho,\xi}\setminus\calH$ goes to a point of $\calH$, the non-symplectic automorphism of the family does not degenerate to a birational non-biregular transformation. What happens instead is that the automorphism degenerates by jumping to another one with a bigger invariant sublattice on a different family. We will show this by describing explicitly 
the new isometry to be considered and by proving that, for a suitable choice of $X$ in the fiber, the birational transformation $\beta$ extends to an automorphism of $X$.

This phenomenon is very common in the context of degenerations of automorphisms of K3 surfaces. Take for instance an ample $\langle 2\rangle$-polarized K3 surface $\Sigma$. It is well-known that $\Sigma$ carries a non-symplectic involution $\iota$ with invariant lattice  
isometric to $\langle 2\rangle$ and such that $\Sigma/\iota\cong\proj 2$. This exhibits $\Sigma$ as a ramified double cover of $\proj 2$ branched
along a smooth sextic curve. If the sextic acquires ADE singularities then the double cover is singular and its minimal resolution $\tilde \Sigma$ is a K3 surface
containing several $(-2)$-curves coming from the desingularisation of the singular points.  In this case the K3 surface $\tilde \Sigma$ carries again a non-symplectic involution but its invariant lattice has higher rank.
A concrete example is when the sextic has a node: generically, the surface~$\tilde \Sigma$ has a N\'eron--Severi lattice isometric to $\langle 2\rangle\oplus \langle -2\rangle$, which is also the invariant lattice of the non-symplectic involution (see~\cite{AN}). What happens is that as soon as the sextic admits singular points, the K3 surface $\tilde \Sigma$ contains rational curves which do not meet the class of square $2$ (the pullback of the class of a line of~$\IP^2$ is not an ample polarization any more, it is only big and nef). Then $\tilde \Sigma$ admits a non-symplectic involution $\tilde{\iota}$ which is of ``different kind'' than $\iota$, since their invariant lattices are different, so heuristically the automorphism survives by jumping to a different family.

\subsection{Degeneracy lattices}

For any $\omega\in\calH$ and $(X,\eta,\iota)\in\cP_{\langle 6\rangle}^{-1}(\omega)$, the \emph{degeneracy lattice} of $(X,\eta,\iota)$ is the sublattice 
of $S(-1)$ generated by MBM classes of $S(-1)$ which are orthogonal to $\omega$ (see the proof of Theorem~\ref{th_isom_mod_spaces}). Here this lattice is simply 
the lattice generated by the roots $\delta$ of $S(-1)$, with $\delta^2=-2$, such that $\omega\in\delta^\perp$. This lattice is clearly globally $\rho$-invariant 
and orthogonal to $j(\langle 6\rangle)$. Generically, $\omega$~belongs to only one hyperplane $\delta^\perp\cap\IP(S(-1)_\xi)$ so the general degeneracy 
lattice is $R_\delta\coloneqq\Span(\delta,\rho(\delta))$. 

It is easy to compute that $R_\delta$ is isometric to the lattice $A_2(-1)$ and that $\rho$~acts trivially on the discriminant
 group~$D_{R_\delta}$. The lattice $R_\delta$ is primitive in~$S(-1)$. Indeed, using the $\cE$-lattice structure (see Section~\ref{ss_hermitian}) of $S(-1)$ we have an isomorphism of $\cE$-modules 
$\frac{S(-1)}{R_\delta}\cong\frac{\cE^{\oplus 11}}{\cE\delta}$. 
Since $\delta$ is a root, with $H_S(\delta)=3$, it is easy to check that the rank one $\cE$-submodule $\cE\delta$ is primitive in $\cE^{\oplus 11}$, so $R_\delta$ is primitive in $S(-1)$. 
Using a result of Nikulin~\cite[Proposition~1.15.1]{Nikulin} we find that two degeneration situations appear, depending on the two possible isometry classes of the orthogonal complement of $R_\delta$ in $S(-1)$: 
$$
R_\delta^\perp\cong U^{\oplus 2}\oplus E_8(-1)^{\oplus 2}
\text{ or } 
R_\delta^\perp\cong U\oplus U(3)\oplus E_8(-1)^{\oplus 2}.
$$
Recall that following~\cite{ACT} a root $\delta$ is called \emph{nodal} if $H_S(S,\delta)=\theta\cE$ and \emph{chordal} if $H_S(S,\delta)=3\cE$. The relation with the above dichotomy is explained by the following observation.

\begin{lemma}\label{lem_chordal_unimodular}
A root $\delta$ is chordal if and only if $R_\delta^\perp$ is unimodular.
\end{lemma}

\begin{proof}
It is easy to check that $H_S(R_\delta,\delta)=3\cE$ and $H_S(R_\delta^\perp,\delta)=0$. If $R_\delta^\perp$ is unimodular, then $R_\delta\oplus R_\delta^\perp=S$ so $H_S(S,\delta)=3\cE$ and
$\delta$ is chordal. Otherwise $\frac{S}{R_\delta\oplus R_\delta^\perp}\cong \frac{\IZ}{3\IZ}$ is generated by a class $x\coloneqq\frac{1}{3}(\delta-\rho(\delta)+y)$ with $y\in R_\delta^\perp$
and we get $H_S(x,\delta)=\theta(1-\xi)\in \theta\cE^\ast$ so $H_S(S,\delta)=\theta \cE$ and $\delta$ is nodal.
\end{proof}

\subsection{Degeneration of the representation}

Let $\delta\in S(-1)$ be a root ($\delta^2=-2$).
Recall that we have choosen primitive embeddings $j\colon \langle 6\rangle\hookrightarrow L$ and $S(-1)\subset L$ such that $j(\langle 6\rangle)^\perp=S(-1)$. 
This induces an embedding  in~$L$ of $\langle 6\rangle\oplus A_2(-1)\cong \langle 6\rangle\oplus R_\delta$, not necessarily primitive. 
We thus consider the saturation  $T_\delta$ of $\langle 6\rangle\oplus R_\delta$ in~$L$, which is the minimal primitive sublattice of~$L$ containing $j(\langle 6\rangle)\oplus R_\delta$.
Observe that the orthogonal complement $S_\delta$ of $T_\delta$ in $L$ is nothing else than the orthogonal complement $R_\delta^\perp$ of $R_\delta$ in $S(-1)$ embedded in $L$.

\begin{lemma} \label{lem_saturation} Let $\delta\in S(-1)$ be a root.
\begin{enumerate}
\item If $\delta$ is nodal, then $\langle 6\rangle\oplus R_\delta$ is primitive in $L$ and $T_\delta\cong U(3)\oplus \langle -2\rangle$.
\item If $\delta$ is chordal, then $\langle 6\rangle\oplus R_\delta$ is not primitive in $L$ and $T_\delta\cong U\oplus \langle -2\rangle$.
\end{enumerate}
\end{lemma}

\begin{proof}
Since $\langle 6\rangle\oplus R_\delta$ has discriminant $3^2\cdot 2$, its only possible saturation is by adding a $3$-divisible class, so either $\disc(T_\delta)=3^2\cdot 2$ (primitive case), or $\disc(T_\delta)=2$ (nonprimitive case).
Recall that if $M$ is a nondegenerate lattice and $N\subset M$ a primitive sublattice, denoting by $\phi\colon M\to N^\ast$ the natural map, we have the following formula (see~\cite[p.47]{GHS_handbook}):
$$
\disc(N)\cdot \disc(M)=[N^\ast:\phi(M)]^2 \cdot \disc(N^\perp).
$$
Applying this formula to $M=L$ and either $N=T_\delta$ or $N=S_\delta$, we get that if $\delta$~is nodal, then necessarily $\disc(T_\delta)=3^2\cdot 2$, and if $\delta$ is chordal, then $\disc(T_\delta)=2$.
If $\delta$~is nodal, the isomorphism $\langle 6\rangle\oplus A_2(-1)\cong U(3)\oplus \langle -2\rangle$ is an easy exercise. If $\delta$ is chordal, the lattice $T_\delta$ is thus an even lattice of discriminant~$2$ and signature~$(1,2)$. By Nikulin's classification, such a lattice is unique up to isometry so ${T_\delta\cong U\oplus\langle -2\rangle}$.
\end{proof}

The representation $\IZ/3\IZ\to\Orth(L),1\mapsto\rho$
has $j(\langle 6\rangle)$ as fixed sublattice and $S(-1)$ as orthogonal complement. We modify this representation as follows. Recall that~$r_v$ denotes the reflection by~$v^\perp$. 

\begin{proposition}\label{prop_rho_delta}
The isometry $\id_{T_\delta}\oplus \rho_{|S_\delta}\in\Orth(T_\delta)\oplus\Orth(S_\delta)$
extends uniquely to an isometry $\rho_\delta\in\Orth(L)$ with fixed sublattice $T_\delta$. One has $\rho_\delta=r_{\rho(\delta)}\circ r_{\delta}\circ\rho$
and $\rho_\delta\in\Mon^2(L)$.
\end{proposition}

\begin{proof}
We apply Lemma~\ref{lem_chordal_unimodular}.
If $\delta$ is a chordal root, then $S_\delta$ is unimodular so the result is clear. 
If $\delta$ is nodal, using~\cite[Proposition~1.15.1]{Nikulin} we know that the discriminant group $D_{S_\delta}$
is isometric to the restriction of $(-q_{R_\delta})\oplus q_{S(-1)}$ to the quotient of a subgroup of $D_{R_\delta}\oplus D_{S(-1)}$. 
As observed above, $\rho$ acts trivially on $D_{R_\delta}$, so to show that it acts trivially on $D_{S_\delta}$ it
is enough to show that it acts trivially on $D_{S(-1)}$. By the same argument as above, we can identify $D_{S(-1)}$
with the quotient of a subgroup of $D_{j(\langle 6\rangle)}\oplus D_L$ by changing a sign on the finite quadratic form. 
By contruction $\rho$ acts trivially on the first factor, and it acts also trivially on $D_L\cong\IZ/2\IZ$ since
$\rho$ has order three. So $\rho$ acts trivially on  $D_{S(-1)}$ and the result follows. It is easy to check
that $\rho_\delta=r_{\rho(\delta)}\circ r_{\delta}\circ\rho$. Indeed, both act trivially on $R_\delta$ and on $j(\langle 6\rangle)\subset \delta^\perp$ and act by $\rho$ on $S_\delta=R_\delta^\perp$.
By~\cite[Section~9.1, Theorem~9.1]{Markman_Torelli}, this shows that $\rho_\delta\in\Mon^2(L)$.
\end{proof}

\subsection{Nodal degenerations}

By~\cite[Theorem~6.1]{ACT}, the period map for cubic threefolds 
$$
\cP_S^{\rho_0,\xi}\colon \cC^\smooth_3\to
\frac{\IB^{10}\setminus(\calH_\nodal\cup \calH_\chordal)}{\IP\Gamma}
$$ 
extends to an isomorphism $\cP_S^{\rho_0,\xi}\colon \cC^\stable_3\to
\frac{\IB^{10}\setminus\calH_\chordal}{\IP\Gamma}$ mapping bijectively
the \emph{stable discriminant locus} $\Delta^\stable_3\coloneqq \cC^\stable_3\setminus \cC^\smooth_3$ to the $\IP\Gamma$-orbit of the nodal hyperplanes arrangement~$\calH_\nodal$. 
We recall here some aspects of their work needed in the sequel.

A general point in $\Delta^\stable_3$ is a cubic threefold $C$ with one nodal (or $A_1$) singular point. In suitable coordinates, the equation of~$C$ writes
$$
f(x_0,\ldots,x_4)\coloneqq x_0q(x_1,\ldots,x_4)+g(x_1,\ldots,x_4)=0
$$
where $q$ is a nondegenerate quadric surface and $g$ is a homogeneous cubic polynomial. By genericity we may assume that the locus $\{g=0\}$ is smooth and that it meets the locus $\{q=0\}$ transversally. The point $p_0=(1,0,0,0,0)\in\proj 4$ is then the node of $C$. We put $k(x_1,\ldots,x_5)\coloneqq g(x_1,\ldots,x_4)+x_5^3$, so that the cuspidal cyclic cubic fourfold $Y$ associated to $C$ has equation
$$
F(x_0,\ldots,x_5)\coloneqq x_0q(x_1,\ldots,x_4)+k(x_1,\ldots,x_5)
$$
with an $A_2$-singularity at $p=(1,0,0,0,0,0)\in\proj 5$. 

We denote by $H\subset\proj 5$ the hyperplane of equation $x_0=0$ and by $Q$, \resp $K$, the zero locus of $q$, \resp $k$, in $H$. By genericity, the intersection $\Sigma_\nodal\coloneqq Q\cap K\subset H$ is a smooth K3 surface. 
By projecting $Y\setminus\{p\}$ to~$H$ through the point~$p$, it is easy to check that the 
subvariety $\Fano{Y,p}\subset\Fano Y$ defined by the lines~$\ell$ of~$Y$ passing through~$p$ maps biholomorphically to~$\Sigma_\nodal$, sending~$\ell$ to~$\ell\cap H$ (see \cite[\S 7]{CG}). 
The Fano variety of lines $\Fano Y$ is a singular holomorphic symplectic variety whose singular locus is contained in the surface $\Fano{Y,p}$ (see~\cite[Proposition~3.5, Theorem~3.6]{Lehn}) and it is birational to $\Sigma_\nodal^{[2]}$. To see it, first observe that
the Cartier divisor $\{q=0\}\cap Y$ in~$\proj 5$ is the cone over~$\Sigma_\nodal$ with vertex~$p$. A general line $\ell\subset Y$ intersects the quadric $\{q=0\}$ in two points which project to two points of~$\Sigma_\nodal$, 
yielding a rational map $\mu\colon\Fano Y\to \Sigma_\nodal^{[2]}$. The birational inverse is given as follows: given two distinct points $x,y\in \Sigma_\nodal$, the plane $\Span(x,y,p)$ cuts~$Y$ along the degenerate cubic curve 
which is the union of the lines $\Span(p,x)$, $\Span(p,y)$ and a third line $\ell\in\Fano Y$ given by the image of $\{x,y\}\in \Sigma_\nodal^{[2]}$.

The covering automorphism $\sigma$ restricts on $H$ to an order three non-symplectic automorphism~$\tau_\nodal$ of~$\Sigma_\nodal$.
 This family of K3 surfaces has been studied by Artebani--Sarti~\cite[Proposition~4.7]{AS}: 
the automorphism~$\tau_\nodal$ fixes a genus $4$ curve which, following the discussion above, is isomorphic to the Fano variety $\Fano{C,p_0}$ of lines though~$p_0$. The invariant lattice in $\HH^2(\Sigma_\nodal,\IZ)$ 
is isometric to~$U(3)$ and its orthogonal complement is isometric to~$U\oplus U(3)\oplus E_8(-1)^{\oplus 2}$.

Let $\delta\in S(-1)$ be a nodal root (a precise choice is unnecessary since the group~$\Gamma$ acts transitively on the set of nodal roots by~\cite[Lemma~6.2]{ACT}). Recall that by Lemma~\ref{lem_saturation}, the lattice $T_\delta$ is isometric to $U(3)\oplus\langle -2\rangle$. An elementary computation shows that the second summand is generated by the $(-2)$-class 
$$
e_\nodal\coloneqq-\vartheta+2\delta+2\rho(\delta),
$$ 
where $\vartheta$ is as above the square $6$ polarization. Observe that the orthogonal complement of $\IZ e_\nodal$ in $L$ is isometric to the K3 lattice $U^{\oplus 3}\oplus E_8(-1)^{\oplus 3}$. We thus obtain an embedding $U(3)\hookrightarrow T_\delta\subset L$ whose image is contained in $(\IZ e_\nodal)^\perp$ and whose orthogonal complement in~$(\IZ e_\nodal)^\perp$ is isometric to $S_\delta=U\oplus U(3)\oplus E_8(-1)^{\oplus 2}$. The isometry $\id_{U(3)}\oplus {\rho_\delta}_{|S_\delta}$ extends uniquely to an isometry $\widetilde{\rho_\delta}$ of~$(\IZ e_\nodal)^\perp$ which is the 
restriction of $\rho_\delta$ to $(\IZ e_\nodal)^\perp\subset L$.

\begin{lemma}\label{lem_tau_is_rho_delta_nodal} Let $\delta\in S(-1)$ be a nodal root.
The action  of~$\tau_\nodal$ on~$\HH^2(\Sigma_\nodal,\IZ)$ is conjugated to the isometry  $\widetilde{\rho_\delta}$ of~$(\IZ e_\nodal)^\perp$.
\end{lemma}

\begin{proof}
In a similar manner as in Section~\ref{ss_modnonsympl}, we introduce the moduli space
$\cK_{U(3)}^{\widetilde{\rho_\delta},\xi}$ parame\-tri\-zing K3 surfaces with a non-symplectic  automorphism of order three whose action on cohomology is conjugated to $\widetilde{\rho_\delta}$, with a period map 
$$
\cP_{U(3)}^{\widetilde{\rho_\delta},\xi}\colon \cK_{U(3)}^{\widetilde{\rho_\delta},\xi}\longrightarrow \Omega_{U(3)}^{\widetilde{\rho_\delta},\xi}\coloneqq \{\omega\in\IP((S_\delta)_{\xi})\,|\, h_{S_\delta}(\omega,\omega)>0\},
$$
where $(S_\delta)_\xi$ is the eigenspace of $\rho_\delta$ on $(S_\delta)\otimes_\IZ\IC$ for the eigenvalue $\xi$ and $h_{S_\delta}$ is the Hermitian product induced by the bilinear form on $S_\delta$.
As observed in the proof of~Proposition~\ref{prop_rho_delta}, the isometry $\rho_\delta$ acts trivially on the discriminant group $D_{S_\delta}$, so $\widetilde{\rho_\delta}$ does also, hence following~\cite[Section~3, Proposition~3.1]{AS} there exists a K3 surface with a non-symplectic automorphism of order $3$ whose action on the second cohomology space is conjugated to $\widetilde{\rho_\delta}$. As explained in $\loccit$, the fixed locus of this automorphism, which depends only on lattice theoretical invariants, consists of one genus $4$ curve. By considering the linear system associated to this curve, Artebani--Sarti~\cite[Section~4, Proposition~4.7]{AS} prove by geometric arguments that this K3 surface with this automorphism belongs to the same family as the K3 surface $\Sigma_\nodal$ with automorphism $\tau_\nodal$. This shows that the action of~$\tau_\nodal$ on~$\HH^2(\Sigma_\nodal,\IZ)$ is conjugated to~$\widetilde{\rho_\delta}$.
\end{proof}

We consider similarly the natural automorphism $\tau_\nodal^{[2]}$ induced on $\Sigma_\nodal^{[2]}$ by $\tau_\nodal$, which satisfies the relation $\mu\circ\sigma=\tau_\nodal^{[2]}\circ\mu$ whenever $\mu\colon\Fano Y\to \Sigma_\nodal^{[2]}$~is defined. Its invariant lattice is~$T_\delta$ so it defines a point in the moduli space $\cN_{T_\delta}^{\rho_\delta,\xi}$ parametrizing irreducible holomorphic symplectic manifolds deformation equivalent to the Hilbert square of a K3 surface, with an order three automorphism whose action on cohomology is conjugated to $\rho_\delta$. Here we have implicitly fixed the primitive embedding $j_\delta \colon T_\delta\hookrightarrow L$ discussed above. Following~\cite[Section~3]{BCS_ball},
since the period map for this moduli space is not injective, we introduce a smaller moduli space containing the point $(\Sigma_\nodal^{[2]},\tau_\nodal^{[2]})$ as follows. Take a marking $\eta\colon \HH^2(\Sigma_\nodal^{[2]},\IZ)\to L$. The image by $\eta$ of the invariant positive cone $(C_{\Sigma_\nodal^{[2]}})^{\tau_\nodal}$ of $\Sigma_\nodal^{[2]}$ defines a connected component~$C_{T_\delta}$ of the positive cone of $T_\delta\otimes\IR$. The hyperplanes $\delta^\perp$, where $\delta\in T_\delta$ is an MBM class, cut out a chamber decomposition of $C_{T_\delta}$. One of these chambers, denoted~$K(T_\delta)$, contains the image by~$\eta$ of the invariant K\"ahler cone  $(K_{\Sigma_\nodal^{[2]}})^{\tau_\nodal}$ of~$\Sigma_\nodal^{[2]}$. We denote by~$\cN_{K(T_\delta)}^{\rho_\delta,\xi}$ the set of points $(X,\eta,\sigma,\iota)$ in $\cN_{T_\delta}^{\rho_\delta,\xi}$ which are \emph{$K(T_\delta)$-general}, meaning that $\eta(K_X^\sigma)=K(T_\delta)$. When the surface~$\Sigma_\nodal$ considered in our 
geometric construction above is a very general point of the moduli space~$\cK_{U(3)}^{\widetilde{\rho_\delta},\xi}$, we may assume that the N\'eron--Severi group $\NS(\Sigma_\nodal^{[2]})$ is isometric to $T_\delta$ (otherwise stated $\NS(\Sigma_\nodal)\cong U(3)$), so by Lemma~\cite[Lemma~5.0.13]{BCS_ball} the point $(\Sigma_\nodal^{[2]},\tau_\nodal^{[2]})$ is $K(T_\delta)$-general and lives in $\cN_{K(T_\delta)}^{\rho_\delta,\xi}$.

\begin{proposition}\label{prop_birat_nodal} Let $\delta\in S(-1)$ be a nodal root.
The stable discriminant locus~$\Delta^\stable_3$ is birational to the moduli space $\cN_{K(T_\delta)}^{\rho_\delta,\xi}$.
\end{proposition}

\begin{proof}
Since the isometry $\rho_0$ is the restriction of~$\rho$ to $S(-1)$ and since $\rho_\delta$ is the restriction of $\rho$ to $S_\delta\subset S(-1)$, we have an inclusion of eigenspaces $(S_\delta)_\xi\subset S_\xi$. If $\omega\in \delta^\perp\cap \Omega_S^{\rho_0,\xi}$, it is certainly also orthogonal to $\rho(\delta)$, so the period $\omega$ lives in the $9$-dimensional complex ball $\Omega_{U(3)}^{\widetilde{\rho_\delta},\xi}=\delta^\perp\cap \Omega_S^{\rho_0,\xi}\subset\calH_\nodal$. By~\cite[Theorem~9.1]{AST}, the period map considered in the proof of Lemma~\ref{lem_tau_is_rho_delta_nodal} induces a bijection
$$
\cP_{U(3)}^{\widetilde{\rho_\delta},\xi}\colon \cK_{U(3)}^{\widetilde{\rho_\delta},\xi}\longrightarrow \frac{\Omega_{U(3)}^{\widetilde{\rho_\delta},\xi}\setminus\calH_{U(3)}}{\Gamma_{U(3)}^{\widetilde{\rho_\delta},\xi}}
$$ 
where $\calH_{U(3)}\coloneqq \bigcup\limits_{\substack{\varepsilon\in S_\delta\\\varepsilon^2=-2}}\varepsilon^\perp$. The group $\Gamma_{U(3)}^{\widetilde{\rho_\delta},\xi}$ is the set of isometries of the K3 lattice~$(\IZ e_\nodal)^\perp$ which commute with $\widetilde{\rho_\delta}$. We thus clearly have an equality of orbit spaces 
$$
\frac{\Omega_{U(3)}^{\widetilde{\rho_\delta},\xi}}{\Gamma_{U(3)}^{\widetilde{\rho_\delta},\xi}}=\frac{\delta^\perp\cap\Omega_{S}^{\rho_0,\xi}}{\Gamma_S^{\rho_0,\xi}}.
$$

The extension of the period map $\cP_S^{\rho_0,\xi}\colon \cC^\smooth_3\to
\frac{\IB^{10}\setminus(\calH_\nodal\cup \calH_\chordal)}{\IP\Gamma}$ 
to a general nodal cubic~$C$ is defined in~\cite{ACT} by using the period of the associated
K3 surface $\Sigma_\nodal$:
$$
\cP_S^{\rho_0,\xi}(C)\coloneqq \cP_{U(3)}^{\widetilde{\rho_\delta},\xi}(\Sigma_\nodal)\in\calH_\nodal.
$$
By studying the limit Hodge structure of the nodal degeneration of a cubic threefold, Allcock--Carlson--Toledo~\cite[Chapter~6]{ACT} show that this definition extends holomorphically $\cP_S^{\rho_0,\xi}$ to an isomorphism  between $\cC^\stable_3$ and
$\frac{\IB^{10}\setminus\calH_\chordal}{\IP\Gamma}$ mapping bijectively $\Delta^\stable_3$ to 
the divisor $\calH_\nodal/\IP\Gamma$.

Turning back to the automorphism $\tau_\nodal^{[2]}$ on $\Sigma_\nodal^{[2]}$, which gives a point in the moduli space~$\cN_{K(T_\delta)}^{\rho_\delta,\xi}$, the period map constructed in~\cite[Theorem~5.0.17]{BCS_ball} produces a bijection
$$
\cP_{T_\delta}^{\rho_\delta,\xi}\colon\cN_{K(T_\delta)}^{\rho_\delta,\xi}\longrightarrow \frac{\Omega_{T_\delta}^{\rho_\delta,\xi}\setminus\left(\calH_{T_\delta}\cup\calH'_{T_\delta}\right)}{\Gamma_{T_\delta}^{\rho_\delta,\xi}}.
$$
The notation is similar to those of Section~\ref{ss_modnonsympl}, it is easy to check that 
$\Omega_{T_\delta}^{\rho_\delta,\xi}=\Omega_{U(3)}^{\widetilde{\rho_\delta},\xi}$
and that $\Gamma_{T_\delta}^{\rho_\delta,\xi}=\Gamma_{U(3)}^{\widetilde{\rho_\delta},\xi}$.
The hyperplane arrangements are
$$
\calH_{T_\delta}= \bigcup\limits_{\substack{\varepsilon\in S_\delta\\\varepsilon^2=-2}}\varepsilon^\perp=\calH_{U(3)}
$$ 
and $\calH'_{T_\delta}$ is the union of the hyperplanes $\varepsilon^\perp$ which meet $K(T_\delta)$, for those MBM classes $\varepsilon\in L$ which admit a decomposition $\varepsilon=\varepsilon_1+\varepsilon_2$ with $\varepsilon_1\in T_\delta\otimes_\IZ \IQ$, 
${\varepsilon_2\in S_\delta\otimes_\IZ \IQ}$, $\varepsilon_1^2<0$, and $\varepsilon_2^2<0$.

Since $\cP_{U(3)}^{\widetilde{\rho_\delta},\xi}$ and $\cP_{T_\delta}^{\rho_\delta,\xi}$ are bijections and $\frac{\Omega_{U(3)}^{\widetilde{\rho_\delta},\xi}\setminus\calH_{U(3)}}{\Gamma_{U(3)}^{\widetilde{\rho_\delta},\xi}}$ is clearly birational to $\frac{\Omega_{T_\delta}^{\rho_\delta,\xi}\setminus\left(\calH_{T_\delta}\cup\calH'_{T_\delta}\right)}{\Gamma_{T_\delta}^{\rho_\delta,\xi}}$, it follows that the moduli spaces  $\Delta^\stable_3$ and $\cK_{T_\delta}^{\rho_\delta,\xi}$, which are birationally equivalent by~\cite[Proposition~5.3]{CMJL}, are also birational to~$\cN_{K(T_\delta)}^{\rho_\delta,\xi}$.
\end{proof}

\begin{remark}  Let us denote by $\kM_4$ the moduli space of genus $4$ curves. From the proof of Proposition~\ref{prop_birat_nodal}, we get birational isomorphisms
$$
\Delta^\stable_3 \sim \cK_{U(3)}^{\widetilde{\rho_\delta},\xi}\sim  \cN_{K(T_\delta)}^{\rho_\delta,\xi}\sim\kM_4
$$
obtained by mapping a general nodal cubic threefold $C$ with node $p_0$ and associated K3 surface $\Sigma_\nodal$ respectively to $\Sigma_\nodal,\Sigma_\nodal^{[2]}$ and $\Fano{C,p_0}$.

On a very general period point $\omega\in\calH_\nodal$, the birational equivalence between $\Delta^\stable_3$ and $\cK_{T_\delta}^{\rho_\delta,\xi}$ can be seen as follows: $\omega$ corresponds to a general nodal cubic $C$ whose uniquely associated K3 surface $\Sigma_n$ is smooth. 
Conversely, the ramified triple cover of $\proj 1\times\proj 1$ branched along a smooth non-hyperelliptic genus $4$ curve of bidegree $(3,3)$ is a K3 surface living in $\cK_{U(3)}^{\widetilde{\rho_\delta},\xi}$, from which the nodal cubic threefold can be recovered (see~Kond\=o~\cite[Theorem~1]{Kondo}). 
\end{remark}

\begin{remark}\label{rem_relate_nodal} To relate with the general statement given in Proposition~\ref{prop_birational_map}, we observe that if we take a very general period point $\omega\in\calH_\nodal$, then the variety $\Sigma_\nodal^{[2]}$ lives in the fiber $\cP_{\langle 6\rangle}^{-1}(\omega)$ and the
birational map $\beta=\tau_\nodal^{[2]}$ is everywhere defined. Observe that $\beta^\ast=(\tau_\nodal^{[2]})^\ast$ is conjugated to $\rho_\delta$. Indeed, both have the same action on the K3 lattice $(\IZ e_\nodal)^\perp$ and act trivially on the class $e_\nodal$. Denoting by $\eta$ a marking, the element $w\in W_\Exc(\Sigma_\nodal^{[2]})$ is here $w=r_{\eta^{-1}(\delta)}\circ r_{\eta^{-1}\rho(\delta)}$.
\end{remark}

\subsection{Chordal degenerations}

Consider the Veronese embedding $\nu_4\colon\proj 1\to\proj 4$, whose image is the standard
rational normal quartic curve $\cR$. Its secant variety ${T\coloneqq\rm{Sec}(\cR)}$ is a cubic hypersurface in $\proj 4$, called a \emph{chordal cubic threefold}, which is singular  exactly along the curve $\cR$. By~\cite[Theorem~1.3]{Allcock} the cubic~$T$ is a strictly $\PGL_5(\IC)$-semistable point of the linear system $\ls{\cO_{\proj 4} (3)}$. Since the stabilizer of~$T$ is $\PGL_2(\IC)$, its $\PGL_5(\IC)$-orbit is $21$-dimensional. The closure of this orbit is thus a codimension $13$ locus in $\ls{\cO_{\proj 4} (3)}$. Blowing it up, we get a $\PGL_5(\IC)$-equivariant projective morphism $\cZ\to\ls{\cO_{\proj 4} (3)}$. As explained in~\cite[Chapter~2]{ACT}, the fiber of this blowup over the chordal cubic $T$ can be identified with an unordered sequence of $12$~points in~$\cR$ as follows. Take a pencil $(C_t)_{t\in \IB^1}$ of cubics degenerating to $C_0=T$, such that the pencil is transverse to the orbit of $T$. For any $t\neq 0$, the smooth cubic~$C_t$ cuts the curve~$\cR$ in $12$~points, whose limit for 
$t\to 0$ is an unordered $12$-tuple in~$\cR$. 

We denote by $\calC_3^\stable$ the GIT moduli space of $\PGL_5(\IC)$-stable points in $\cZ$ (as in~\cite[p.27]{ACT} the notation is tricky: $\calC_3^\stable$ 
is not a blowup of $\cC_3^\stable$). By~\cite[Theorem~3.2]{ACT}, the \emph{stable chordal locus} $\Xi^\stable_3\coloneqq \calC_3^\stable\setminus\cC_3^\stable$ can be identified with the  locus in $\cR^{12}/\kS_{12}$ of unordered $12$-tuples with no point of multiplicity greater or equal than $6$. 

A general point $(T,\mu)$ in $\Xi_3^\stable$ is a $12$-tuple $\mu$ of distinct unordered points in~$\cR$. Consider the elliptic K3 surface $\Sigma_\chordal$ over~$\cR$ which has a cuspidal fiber over each point of~$\mu$. This family of K3 surfaces has been studied by Artebani--Sarti~\cite[Proposition~4.2]{AS}: it admits a non-symplectic automorphism~$\tau_\chordal$ of order three, preserving the fibration, whose fixed locus consists of the zero section of the fibration and of a genus~$5$ curve which is the ramified double cover of~$\cR$ branched along the points~$\mu$. The invariant lattice in~$\HH^2(\Sigma_\chordal,\IZ)$ is isometric to~$U$ and its orthogonal complement is isometric to $U^{\oplus 2}\oplus E_8(-1)^{\oplus 2}$.

Let $\delta\in S(-1)$ be a chordal root (as above, an effective choice is  unnecessary). Recall that by Lemma~\ref{lem_saturation}, the saturated lattice $T_\delta$ is isometric to $U\oplus\langle -2\rangle$. An elementary computation shows that the second term is generated by the $(-2)$-class 
$$
e_\chordal\coloneqq \frac{1}{3}(-\vartheta+4\delta+2\rho(\delta)).
$$ 
The orthogonal complement of $\IZ e_\chordal$ in $L$ is isometric to the K3 lattice. We thus obtain an embedding $U\hookrightarrow T_\delta\subset L$ whose image is contained in $(\IZ e_\chordal)^\perp$ and whose orthogonal complement in $(\IZ e_\chordal)^\perp$ is isometric to $S_\delta=U^{\oplus 2}\oplus E_8(-1)^{\oplus 2}$. The isometry $\id_{U}\oplus {\rho_\delta}_{|S_\delta}$ extends uniquely to an isometry $\widetilde{\rho_\delta}$ of~$(\IZ e_\chordal)^\perp$ which is the restriction of $\rho_\delta$ to $(\IZ e_\chordal)^\perp\subset L$.

\begin{lemma}\label{lem_tau_is_rho_delta_chordal} Let $\delta\in S(-1)$ be a chordal root. 
The action  of~$\tau_\chordal$ on~$\HH^2(\Sigma_\chordal,\IZ)$ is conjugated to the isometry  $\widetilde{\rho_\delta}$ of~$(\IZ e_\chordal)^\perp$.
\end{lemma}

\begin{proof} The proof follows the same lines as the proof of Lemma~\ref{lem_tau_is_rho_delta_nodal}, we only point out the necessary changes. Consider the moduli space
$\cK_U^{\widetilde{\rho_\delta},\xi}$ and the period map 
$$
\cP_U^{\widetilde{\rho_\delta},\xi}\colon \cK_U^{\widetilde{\rho_\delta},\xi}\longrightarrow \Omega_U^{\widetilde{\rho_\delta},\xi}\coloneqq \{\omega\in\IP((S_\delta)_{\xi})\,|\, h_{S_\delta}(\omega,\omega)>0\},
$$
and use \cite[Theorem~3.3, Proposition~4.2]{AS} to relate $\tilde\rho_\delta$ to the action of $\tau_\chordal$ on the K3 surface $\Sigma_\chordal$.
\end{proof}

We consider similarly the automorphism $\tau_\chordal^{[2]}$ induced on $\Sigma_\chordal^{[2]}$, with invariant lattice~$T_\delta$. It defines a point in the moduli space $\cN_{T_\delta}^{\rho_\delta,\xi}$. As above we introduce the moduli space~$\cN_{K(T_\delta)}^{\rho_\delta,\xi}$, which contains the point $(\Sigma_\chordal^{[2]},\tau_\chordal^{[2]})$. 

\begin{proposition}\label{prop_birat_chordal} Let $\delta\in S(-1)$ be a chordal root.
The stable chordal locus~$\Xi^\stable_3$ is birational to the moduli space $\cN_{K(T_\delta)}^{\rho_\delta,\xi}$.
\end{proposition}

\begin{proof} 
The proof is similar to that of Proposition~\ref{prop_birat_nodal}, we only indicate the needed changes.
By~\cite[Theorem~9.1]{AST}, the period map considered in the proof of Lemma~\ref{lem_tau_is_rho_delta_chordal} induces a bijection
$$
\cP_U^{\widetilde{\rho_\delta},\xi}\colon \cK_U^{\widetilde{\rho_\delta},\xi}\longrightarrow \frac{\Omega_U^{\widetilde{\rho_\delta},\xi}\setminus\calH_U}{\Gamma_U^{\widetilde{\rho_\delta},\xi}}
$$ 
where $\calH_U\coloneqq \bigcup\limits_{\substack{\varepsilon\in S_\delta\\\varepsilon^2=-2}}\varepsilon^\perp$. The group $\Gamma_U^{\widetilde{\rho_\delta},\xi}$ is the set of isometries of the K3 lattice~$(\IZ e_\chordal)^\perp$ which commute with $\widetilde{\rho_\delta}$. We thus have an equality of orbit spaces 
$$
\frac{\Omega_{U}^{\widetilde{\rho_\delta},\xi}}{\Gamma_{U}^{\widetilde{\rho_\delta},\xi}}=\frac{\delta^\perp\cap\Omega_{S}^{\rho_0,\xi}}{\Gamma_S^{\rho_0,\xi}}.
$$

Allcock--Carlson--Toledo~\cite[Theorem~3.7, Theorem~6.1]{ACT} extend the period map $\cP_S^{\rho_0,\xi}\colon \cC^\stable_3\to
\frac{\IB^{10}\setminus\calH_\chordal}{\IP\Gamma}$ 
to an isomorphism $\cP_S^{\rho_0,\xi}\colon \calC^\stable_3\to
\frac{\IB^{10}}{\IP\Gamma}$ mapping bijectively the stable chordal locus $\Xi^\stable_3$ to the divisor $\calH_\chordal/\IP\Gamma$. This extension is unique by Riemann extension theorem, so we may reinterpret generically their construction by using the period of elliptic K3 surfaces: for a general point $(T,\mu)\in\Xi^\stable_3$ with associated elliptic K3 surface $\Sigma_\chordal$, we have:
$$
\cP_S^{\rho_0,\xi}((T,\mu))\coloneqq \cP_U^{\widetilde{\rho_\delta},\xi}(\Sigma_\chordal)\in\calH_\chordal.
$$

Now consider the isomorphism given by the period map
$$
\cP_{T_\delta}^{\rho_\delta,\xi}\colon\cN_{K(T_\delta)}^{\rho_\delta,\xi}\longrightarrow \frac{\Omega_{T_\delta}^{\rho_\delta,\xi}\setminus\left(\calH_{T_\delta}\cup\calH'_{T_\delta}\right)}{\Gamma_{T_\delta}^{\rho_\delta,\xi}}.
$$
As above, here
$\Omega_{T_\delta}^{\rho_\delta,\xi}=\Omega_U^{\widetilde{\rho_\delta},\xi}$
and $\Gamma_{T_\delta}^{\rho_\delta,\xi}=\Gamma_{U}^{\widetilde{\rho_\delta},\xi}$.
We thus get birational isomorphisms
$$
\Xi^\stable_3 \sim \cK_{U}^{\widetilde{\rho_\delta},\xi}\sim  \cN_{K(T_\delta)}^{\rho_\delta,\xi}\sim\cR^{12}/\kS_{12}.
$$
\end{proof}

\begin{remark}\label{rem_relate_chordal} As in Remark~\ref{rem_relate_nodal}, if we take a general period point $\omega\in\calH_\chordal$ then the variety $\Sigma_\chordal^{[2]}$ lives in the fiber $\cP_{\langle 6\rangle}^{-1}(\omega)$ and the
birational map $\beta=\tau_\chordal^{[2]}$ is everywhere defined. 
\end{remark}


\section{Pfaffian cubics with triple covering}\label{s_pfaff}

Recall that $\cP_S\colon\cC^\smooth_4\to \frac{\Omega_S^\circ}{\Gamma_S}$ is the period map for cubic fourfolds (see Section~\ref{ss_compare}), 
$\Lambda=\langle 1\rangle^{\oplus 21}\oplus \langle -1\rangle^{\oplus 2}$, $\kappa\in\Lambda$ is a polarization class of square $3$ 
such that $S=(\IZ\kappa)^\perp$. For any rank two lattice $M$ primitively embedded in $\Lambda$ and containing $\kappa$, we define a hyperplane 
$$
\cH_M\coloneqq\{\omega\in\Omega_S\,|\,\omega\perp M\}.
$$ 
By Hassett~\cite[Propositions~3.2.2\&3.2.4]{Hassett} (see also Laza~\cite[\S 2]{Laza_period}), for any $d$ the hyperplanes $\cH_M$ where $\det(M)=d$ form a single $\Gamma_S$-orbit. We denote by $\cH_d\subset \frac{\Omega_S^\circ}{\Gamma_S}$ the corresponding irreducible hypersurface, which is nonempty if and only if $d\equiv 0, 2 \mod 6$ and $d>6$ \cite[Theorem~1.0.1]{Hassett}. By Hassett~\cite{Hassett} and Laza~\cite[Theorem~1.1]{Laza_period}, the image of the period map $\cP_S$ is precisely $\frac{\Omega_S^\circ}{\Gamma_S}\setminus(\cH_2\cup \cH_6)$. Of particular interest in this section are the divisor $\cH_8$ parametrizing cubic fourfolds containing a plane, and the divisor $\Pfaff$ of Pfaffian cubic fourfolds, whose closure is $\cH_{14}$. 
Proposition~\ref{prop_pfaff} below shows that  $\Pfaff\setminus\cH_8$ contains a cyclic cubic fourfold. 

Let $V$ be a $6$-dimensional rational vector space and {$(Z\subset W)$} be general subspaces in the flag variety $\Flag(5,6;\wedge^2 V^\ast)$. 
The Pfaffian locus $\Pfaff\subset\IP\left(\wedge^2V^\ast\right)$ is the hypersurface of degenerate skew-symmetric $2$-forms (by convention $\IP(-)$ denotes the projective space of lines).
 To the pair $(Z\subset W)$ we associate the cubic threefold $C\coloneqq \Pfaff\cap\IP(Z)$ and the cubic fourfold $Y\coloneqq \Pfaff\cap\IP(W)$. Using the Pl\"ucker embedding 
we consider the Grassmaniann $\Grass(2,V)$ as the subvariety of decomposable tensors in $\IP(\wedge^2 V)$. The orthogonal flag $W^\circ\subset Z^\circ$ in $\Flag(9,10;\wedge^2 V)$ defines 
a Fano threefold $X\coloneqq \Grass(2,V)\cap\IP(Z^\circ)$ of index one and genus $8$ and a K3 surface $\Sigma\coloneqq \Grass(2,V)\cap\IP(W^\circ)$ of genus $14$ as anticanonical section of $X$ 
(see Iliev--Manivel~\cite{IM}). We  call the pair $(\Sigma,X)$ the \emph{dual} of the pair $(Y,C)$. By Beauville--Donagi~\cite[Proposition~5]{BD}, there exists a birational map between the Fano 
variety of lines $\Fano Y$ and the Hilbert square $\Sigma^{[2]}$, which is an isomorphism if $Y$ contains no plane and if $\Sigma$ contains no line.

\begin{proposition}\label{prop_pfaff}
There exists a smooth Pfaffian cubic fourfold $Y$, defined over~$\IQ$, containing no plane, which is the triple cover of $\proj 4$ along a smooth Pfaffian cubic threefold $C$, and whose dual K3 surface $\Sigma$ is smooth and contains no line.
\end{proposition}

\begin{proof}
Let $V$ be a $6$-dimensional rational vector space. We fix an isomorphism $\wedge^6V^\ast\cong\IQ$. The \emph{Pfaffian} of a skew-symmetric $2$-form $\varphi\in\wedge^2V^\ast$ is ${\Pf(\varphi)=\frac{1}{6}\varphi^3}$.
We fix  a skew-symmetric $2$-form of maximal rank $\varphi_0\in\wedge^2 V^\ast$. Without loss of generality, we can assume that $\Pf(\varphi_0)=1$.
We define on $\wedge^2 V^\ast$ a linear form~$h$ and a quadratic form~$q$ by the decomposition
$$
\Pf(x_0\varphi_0+\varphi)=x_0^3+x_0^2 h(\varphi)+x_0 q(\varphi)+\Pf(\varphi)\quad\forall \varphi\in\wedge^2 V^\ast,
$$
where $h(\varphi)=\frac{1}{2}\varphi_0^2\wedge\varphi$ and $q(\varphi)=\frac{1}{2}\varphi_0\wedge \varphi^2$.
Since $q(\varphi_0)=3$, the hyperplane $H\coloneqq \ker(h)$ is regular for $q$. 
A computation shows that the rational quadratic form~$q$ has Witt index $7$ and that its restriction 
$q_{H}$ to $H$ has Witt index $6$. Let $Z$ be a general $5$-dimensional totally isotropic vector subspace of $H$ and $W\coloneqq \IC\varphi_0\oplus Z$. The cubic fourfold
$$
Y_Z\coloneqq\{ [\varphi]\in\IP(W)\,|\,\Pf(\varphi)=0\}
$$
is by construction a ramified cyclic triple covering of $\IP(Z)$ branched along the cubic threefold 
$$
C_Z\coloneqq\{[\varphi]\in\IP(Z)\,|\,\Pf(\varphi)=0\}.
$$ 
We show that the isotropic space~$Z$ can be
choosen in such a way that the general elements~$Y_Z$ and~$C_Z$ of this family are smooth by exhibiting an example in Appendix~\ref{app_example}. 

Assume that $Y_Z$ contains a plane
$P\subset\IP(W)$. This plane cannot be contained in $\IP(Z)$ otherwise it would be contained in $C_Z$, which would be singular. 
Thus $P$~meets~$\IP(Z)$ along a line $\Span(\varphi_2,\varphi_3)$ with $\varphi_2,\varphi_3\in Z$. Take $\varphi_1\in W\setminus Z$ such that $P=\Span(\varphi_1,\varphi_2,\varphi_3)$ and 
denote by $\varphi'_1\in\IP(Z)$ the projection of $\varphi_1$ on~$\IP(Z)$ from the point~$\varphi_0$. We may assume that $\varphi_1=\varphi_0+\varphi'_1$. For any $x_1,x_2,x_3\in\IC$ we have
\begin{align*}
0&=\Pf(x_1\varphi_1+x_2\varphi_2+x_3\varphi_3)=\Pf(x_1\varphi_0+(x_1\varphi'_1+x_2\varphi_2+x_3\varphi_3))\\
&=x_1^3+\Pf(x_1\varphi'_1+x_2\varphi_2+x_3\varphi_3).
\end{align*}
In particular $\Pf(\varphi'_1)\neq 0$ and $\Pf(x_2\varphi_2+x_3\varphi_3)=0$ for any $x_2,x_3$. This shows that $\varphi'_1\notin C_Z$ and that $P$ meets $C_Z$ along the triple line $\Span(\varphi_2,\varphi_3)$. 

To prove that $C_Z$ has no triple line, we make a local analysis of the Fano surface of lines $\Fano {C_Z}$, following~\cite[\S 1]{Murre}. 
Denote by $p_{i,j}$, $0\leq i<j\leq 4$ the Pl\"ucker coordinates of the Grassmannian of lines $\Grass(1,\IP(Z))\subset\proj 9$ and denote by $x_0,\ldots,x_4$ the coordinates 
on $\IP(Z)$. On the affine chart $p_{0,1}=1$, the variety $\Grass(1,\IP(Z))$ is isomorphic to $\IC^6$ and a point $\ell\in \Grass(1,\IP(Z))$ in this chart corresponds to a line generated by 
two vectors $\ell_0=(1,0,-p_{1,2},-p_{1,3},-p_{1,4})$ and $\ell_1=(0,1,p_{0,2},p_{0,3},p_{0,4})$. Any plane $P\supset\ell$ cuts the plane $x_0=x_1=0$ at a unique point 
$p=(0,0,\alpha_2,\alpha_3,\alpha_4)$ and $P=\Span(\ell,p)$. The plane cubic $P\cap C_Z$ has equation
$F(u\ell_0+v\ell_1+tp)=0$ where $(u:v:t)$ are the projective coordinates of the plane $P$. Expanding in $t$ we get:
$$
F(u\ell_0+v\ell_1)+t\sum_{i=2}^4\frac{\partial F}{\partial x_i}(u\ell_0+v\ell_1)\alpha_i+t^2\sum_{2\leq i,j\leq 4}\frac{\partial^2F}{\partial x_i\partial x_j}(u\ell_0+v\ell_1)\alpha_i\alpha_j+t^3F(p)=0.
$$
The line $\ell$ of equation $t=0$ in $P$ is contained in $C_Z$ if and only if ${F(u\ell_0+v\ell_1)=0}$ for every $u,v$. Write $\frac{\partial F}{\partial x_i}=u^2\phi_i^{2,0}+uv\phi^{1,1}_i+v^2\phi^{0,2}_i$ 
where $\phi^{j,k}_i$ are functions of the local Pl\"ucker coordinates. The plane $P$ is tangent at $C_Z$ along $\ell$ (\ie $P\cap C_Z=2\ell+\ell'$) if and only if $\det\left(\phi_i^{j,k}\right)_{i,j,k}=0$. 
As explained in~\cite[Corollary~1,9]{Murre}, this defines a reducible curve in $\Fano {C_Z}$ with smooth irreducible components (which may intersect with each other). Finally, 
$\ell$ is a triple line (\ie $P\cap C_Z=3\ell$) if and only if the residual plane cubic equation is a multiple of $t^3$. The locus of triple lines is thus easy to 
compute on each local chart of $\Grass(1,\IP(Z))$ and with the help of a computer program we find that $C_Z$ contains none.

To conclude the proof, a computer program shows that the dual K3 surface~$\Sigma$ contains no line: to reduce the computation we use the embedding of $\Sigma$ in $\IP(W^\circ)$ and we compute the equations of its Fano variety of lines $\Fano \Sigma$ in $\IP(\wedge^2(W^\circ))$ to check that it is empty.
\end{proof}

The proof of the global Torelli theorem of Huybrechts--Markman--Verbitsky has made possible to show the existence of hyperk\"ahler manifolds admitting an automorphism 
with a given action on cohomology (see for instance~\cite[Theorem~5.5]{BCS_class}), but it is still hard to construct automorphisms on a given hyperk\"ahler manifold
without deforming it (see \cite{BCNS} for related results on involutions of Hilbert squares of general K3 surfaces). In particular, it is very challenging to show the existence 
of a non-natural automorphism on the Hilbert square of a K3 surface (\ie an automorphism which does not come from an automorphism of the underlying K3 surface, see~\cite{Boissiere}). 
Until now, the only such geometric realization was the involution found by Beauville on the Hilbert square of a general quartic in $\proj 3$. Proposition~\ref{prop_pfaff} gives a 
new geometric example with an order three automorphism.

\begin{corollary}\label{cor_aut}
There exists a smooth complex K3 surface whose Hilbert square admits a non-natural non-symplectic automorphism of order three, with fixed locus isomorphic to the Fano surface of lines on a cubic threefold,
and with invariant lattice isometric to~$\langle 6 \rangle$.
\end{corollary}

\begin{proof}
The cyclic Pfaffian cubic fourfold $Y$ constructed in Proposition~\ref{prop_pfaff} has a non-symplectic automorphism $\sigma$ of order three. Since $Y$ contains no plane and since its dual K3 surface $\Sigma$
contains no line, the isomorphism $\Fano Y\cong \Sigma^{[2]}$ produces a non-symplectic automorphism of order three on $\Sigma^{[2]}$ which we still denote by~$\sigma$. Since the invariant lattice $\HH^2(S^{[2]},\IZ)^\sigma$ 
is isometric to $\langle 6\rangle$, the automorphism $\sigma$ is certainly non-natural (see~\cite[Section~6.1]{BCS_class}).
\end{proof}

Looking at the geometric description of the isomorphism $\Fano Y\cong \Sigma^{[2]}$ in~\cite[proof of Proposition~5]{BD},
we get a geometric description of the action of $\sigma$ on $\Sigma^{[2]}$ as follows. Let $\zeta=\{P,Q\}\in \Sigma^{[2]}$ with $P,Q\in \Grass(2,V)\cap\IP(W^\circ)$ that we assume to be distinct for simplicity.
Consider the $4$-plane $K=P+Q$, the $5$-dimensional linear space $M=\Span(\Grass(2,K))\subset \IP(\wedge^2 V)$ and the line 
$$
\ell_{P,Q}\coloneqq\{\varphi\in\IP(W)\,|\, \varphi_{|K}=0\}\in\Fano Y.
$$
This line is a Kronecker pencil of skew-symmetric $2$-forms whose unique bilagrangian is $K$. The points $P,Q$ can be recovered from $K$ as follows. First observe that ${\ell_{P,Q}=M^\circ\cap\IP(W)}$, then by 
dimension counting $M\cap\IP(W^\circ)$ is a line, which cuts the quadric $\Grass(2,K)$ exactly in the points $P,Q$. If the line $\ell_{P,Q}$ is contained in~$\IP(Z)$, it is invariant by $\sigma$,
so $\zeta$ is invariant by $\sigma$. Otherwise, this line cuts $\IP(Z)$ in a point $\varphi_{P,Q}$. The plane $\Span(\varphi_0,\ell_{P,Q})$ cuts $Y$ along a cubic curve with $\varphi_{P,Q}$ as a triple point,
it is thus the union of three lines passing through $\varphi_{P,Q}$, which are precisely $\ell_{P,Q}, \sigma(\ell_{P,Q})$ and $\sigma^2(\ell_{P,Q})$. We denote the corresponding bilagrangians by $K$, $\sigma K$ and $\sigma^2 K$, which
are three $4$-planes of $V$ meeting along the $2$-plane $\ker(\varphi_{P,Q})$. As explained above, the $4$-plane $\sigma K$ characterizes a length two subscheme of $\Sigma$ which is the image of~$\zeta$ by~$\sigma$.


\appendix

\section{A smooth Pfaffian cubic fourfold with triple covering}
\label{app_example}

Let $M$ be the skew-symmetric matrix of linear forms on $\IQ^6$:
$$
\begin{pmatrix}
0 & u_1 & u_2 & u_3-x_5 & u_4 & u_5\\
-u_1 & 0 & u_6 & u_7 & u_8-x_5 & u_9\\
-u_2 & -u_6 & 0 & u_{10} & u_{11} & u_{12}-x_5\\
-u_3+x_5 & -u_7  & -u_{10} & 0 & u_{13} & u_{14}\\
-u_4 & -u_8+x_5 & -u_{11} & -u_{13} & 0 & u_{15}\\
-u_5 & -u_9 & -u_{12}+x_5 & -u_{14} & -u_{15} & 0
\end{pmatrix}
$$
with:
\begin{align*}
u_1&= 2x_0+3x_1+2x_2+3x_3+2x_4,
& u_2&=x_0+x_2+x_3,\\
u_3&=2x_0+3x_1+2x_2+2x_3+x_4,
&u_4&=3x_0+x_1+2x_2+x_3+x_4,\\
u_5&=2x_0+4x_1+3x_2+3x_3+x_4,
&u_6&=3x_0+4x_1+4x_2+4x_3+x_4,\\
u_7&=x_1+x_2+x_3,
&u_8&=2x_0+4x_1+3x_2+3x_3+x_4,\\
u_9&=2x_0+2x_1+x_2+2x_3+2x_4,
&u_{10}&=2x_1+x_2+2x_3+x_4,\\
u_{11}&=x_0-x_1+x_2-x_3-x_4,
&u_{12}&=-4x_0-7x_1-5x_2-5x_3-2x_4,\\
u_{13}&=-4x_0-4x_1-6x_2-4x_3,
&u_{14}&=6x_0+11x_1+7x_2+8x_3+4x_4,\\
u_{15}&=-4x_0-8x_1-5x_2-5x_3-2x_4. &\\
\end{align*}
The Pfaffian of $A$ defines a cubic fourfold $Y\subset\proj 5$ of equation
\begin{align*}
x_5^3&+\left(-42x_0^3 - 179x_0^2x_1 - 187x_0x_1^2 - 32x_1^3 - 161x_0^2x_2 -391x_0x_1x_2 \right.\\
&- 155x_1^2x_2 - 193x_0x_2^2 - 196x_1x_2^2 - 74x_2^3 
- 138x_0^2x_3 - 280x_0x_1x_3\\
& - 41x_1^2x_3 - 291x_0x_2x_3 
- 205x_1x_2x_3 
- 142x_2^2x_3 - 99x_0x_3^2 - x_1x_3^2\\
& - 62x_2x_3^2 + 9x_3^3 - 46x_0^2x_4 - 104x_0x_1x_4 - 27x_1^2x_4
 - 91x_0x_2x_4 \\
 &- 71x_1x_2x_4 - 36x_2^2x_4 - 80x_0x_3x_4
 - 20x_1x_3x_4 - 40x_2x_3x_4 + 4x_3^2x_4\\
 &\left. - 24x_0x_4^2 - 17x_1x_4^2 - 19x_2x_4^2
  - 11x_3x_4^2 - 3x_4^3 \right)=0
\end{align*}
which is the ramified triple covering of $\proj 4$ branched over a cubic threefold $C$.
It is easy to check that both $C$ and $Y$ are nonsingular. Following the method explained in the proof of Proposition~\ref{prop_pfaff}, we show that the dual K3 surface contains no line.

\bibliographystyle{plain}
\bibliography{biblio}

\begin{thebibliography}{10}

\bibitem{AN}
V.~Alexeev and V.~V. Nikulin.
\newblock {\em Del {P}ezzo and {$K3$} surfaces}, volume~15 of {\em MSJ
  Memoirs}.
\newblock Mathematical Society of Japan, Tokyo, 2006.

\bibitem{Allcock}
D.~Allcock.
\newblock The moduli space of cubic threefolds.
\newblock {\em J. Algebraic Geom.}, 12(2):201--223, 2003.

\bibitem{ACT_surf}
D~Allcock, J.~A. Carlson, and D.~Toledo.
\newblock The complex hyperbolic geometry of the moduli space of cubic
  surfaces.
\newblock {\em J. Algebraic Geom.}, 11(4):659--724, 2002.

\bibitem{ACT}
D.~Allcock, J.~A. Carlson, and D.~Toledo.
\newblock The moduli space of cubic threefolds as a ball quotient.
\newblock {\em Mem. Amer. Math. Soc.}, 209(985):xii+70, 2011.

\bibitem{AV}
E.~Amerik and M.~Verbitsky.
\newblock Rational {C}urves on {H}yperk\"ahler {M}anifolds.
\newblock {\em International Mathematics Research Notices},
  2015(23):13009--13045, 2015.

\bibitem{AS}
M.~Artebani and A.~Sarti.
\newblock Non-symplectic automorphisms of order 3 on {$K3$} surfaces.
\newblock {\em Math. Ann.}, 342(4):903--921, 2008.

\bibitem{AST}
M.~Artebani, A.~Sarti, and S.~Taki.
\newblock {$K3$} surfaces with non-symplectic automorphisms of prime order.
\newblock {\em Math. Z.}, 268(1-2):507--533, 2011.
\newblock With an appendix by Shigeyuki Kond\=o.

\bibitem{BHT}
A.~Bayer, B.~Hassett, and Y.~Tschinkel.
\newblock Mori cones of holomorphic symplectic varieties of {K}3 type.
\newblock {\em Ann. Sci. \'Ec. Norm. Sup\'er. (4)}, 48(4):941--950, 2015.

\bibitem{B_modcub}
A.~Beauville.
\newblock Moduli of cubic surfaces and {H}odge theory (after {A}llcock,
  {C}arlson, {T}oledo).
\newblock In {\em G\'eom\'etries \`a courbure n\'egative ou nulle, groupes
  discrets et rigidit\'es}, volume~18 of {\em S\'emin. Congr.}, pages 445--466.

\bibitem{B_remarks}
A.~Beauville.
\newblock Some remarks on {K}\"ahler manifolds with {$c_{1}=0$}.
\newblock In {\em Classification of algebraic and analytic manifolds ({K}atata,
  1982)}, volume~39 of {\em Progr. Math.}, pages 1--26.

\bibitem{B_monodromy}
A.~Beauville.
\newblock Le groupe de monodromie des familles universelles d'hypersurfaces et
  d'intersections compl\`etes.
\newblock In {\em Complex analysis and algebraic geometry ({G}\"ottingen,
  1985)}, volume 1194 of {\em Lecture Notes in Math.}, pages 8--18. Springer,
  Berlin, 1986.

\bibitem{BD}
A.~Beauville and R.~Donagi.
\newblock La vari\'et\'e des droites d'une hypersurface cubique de dimension
  {$4$}.
\newblock {\em C. R. Acad. Sci. Paris S\'er. I Math.}, 301(14):703--706, 1985.

\bibitem{Boissiere}
S.~Boissi\`ere.
\newblock Automorphismes naturels de l'espace de {D}ouady de points sur une
  surface.
\newblock {\em Canad. J. Math.}, 64(1):3--23, 2012.

\bibitem{BCS_ball}
S.~Boissi\`ere, C.~Camere, and A.~Sarti.
\newblock Complex ball quotients from manifolds of ${K3}^{[n]}$-type.
\newblock \texttt{ArXiv:1512.02067}.

\bibitem{BCS_class}
S.~Boissi\`ere, C.~Camere, and A.~Sarti.
\newblock Classification of automorphisms on a deformation family of
  hyper-{K}{\"a}hler four-folds by {$p$}-elementary lattices.
\newblock {\em Kyoto J. Math.}, 56(3):465--499, 2016.

\bibitem{BCNS}
S.~Boissi\`ere, A.~Cattaneo, M.~Nieper-Wisskirchen, and A.~Sarti.
\newblock The automorphism group of the {H}ilbert scheme of two points on a
  generic projective {K}3 surface.
\newblock In {\em K3 surfaces and their moduli}, volume 315 of {\em Progr.
  Math.}, pages 1--15. Birkh\"auser/Springer, [Cham], 2016.

\bibitem{CMJL}
S.~Casalaina-Martin, D.~Jensen, and R.~Laza.
\newblock The geometry of the ball quotient model of the moduli space of genus
  four curves.
\newblock In {\em Compact moduli spaces and vector bundles}, volume 564 of {\em
  Contemp. Math.}, pages 107--136. Amer. Math. Soc., Providence, RI, 2012.

\bibitem{CG}
C.~H. Clemens and P.~A. Griffiths.
\newblock The intermediate {J}acobian of the cubic threefold.
\newblock {\em Ann. of Math. (2)}, 95:281--356, 1972.

\bibitem{DM}
O.~Debarre and E.~Macr{\`i}.
\newblock Unexpected isomorphisms between hyperk{\"a}hler fourfolds.
\newblock \texttt{arXiv:1704.01439v1}.

\bibitem{DvGK}
I.~Dolgachev, B.~van Geemen, and S.~Kond\=o.
\newblock A complex ball uniformization of the moduli space of cubic surfaces
  via periods of {$K3$} surfaces.
\newblock {\em J. Reine Angew. Math.}, 588:99--148, 2005.

\bibitem{DK}
I.~V. Dolgachev and S.~Kond\=o.
\newblock Moduli of {$K3$} surfaces and complex ball quotients.
\newblock In {\em Arithmetic and geometry around hypergeometric functions},
  volume 260 of {\em Progr. Math.}, pages 43--100. Birkh\"auser, Basel, 2007.

\bibitem{GH}
P.~Griffiths and J.~Harris.
\newblock {\em Principles of algebraic geometry}.
\newblock Wiley Classics Library. John Wiley \& Sons, Inc., New York, 1994.
\newblock Reprint of the 1978 original.

\bibitem{GHS_moduli}
V.~Gritsenko, K.~Hulek, and G.~K. Sankaran.
\newblock Moduli spaces of irreducible symplectic manifolds.
\newblock {\em Compos. Math.}, 146(2):404--434, 2010.

\bibitem{GHS_handbook}
V.~Gritsenko, K.~Hulek, and G.~K. Sankaran.
\newblock Moduli of {K}3 surfaces and irreducible symplectic manifolds.
\newblock In {\em Handbook of moduli. {V}ol. {I}}, volume~24 of {\em Adv. Lect.
  Math. (ALM)}, pages 459--526. Int. Press, Somerville, MA, 2013.

\bibitem{Hassett}
B.~Hassett.
\newblock Special cubic fourfolds.
\newblock {\em Compositio Math.}, 120(1):1--23, 2000.

\bibitem{Huybrechts_Invent}
D.~Huybrechts.
\newblock Compact hyper-{K}{\"a}hler manifolds: basic results.
\newblock {\em Invent. Math.}, 135(1):63--113, 1999.

\bibitem{IM}
A.~Iliev and L.~Manivel.
\newblock Pfaffian lines and vector bundles on {F}ano threefolds of genus 8.
\newblock {\em J. Algebraic Geom.}, 16(3):499--530, 2007.

\bibitem{Kondo}
S.~Kond\=o.
\newblock The moduli space of curves of genus 4 and {D}eligne-{M}ostow's
  complex reflection groups.
\newblock In {\em Algebraic geometry 2000, {A}zumino ({H}otaka)}, volume~36 of
  {\em Adv. Stud. Pure Math.}, pages 383--400. Math. Soc. Japan, Tokyo, 2002.

\bibitem{KudlaRapoport}
S.~Kudla and M.~Rapoport.
\newblock On occult period maps.
\newblock {\em Pacific J. Math.}, 260(2):565--581, 2012.

\bibitem{Laza_period}
R.~Laza.
\newblock The moduli space of cubic fourfolds via the period map.
\newblock {\em Ann. of Math. (2)}, 172(1):673--711, 2010.

\bibitem{Lehn}
Christian Lehn.
\newblock Twisted cubics on singular cubic fourfolds - on {S}tarr's fibration.
\newblock \texttt{arXiv:1504.06406}.

\bibitem{LS}
E.~Looijenga and R.~Swierstra.
\newblock The period map for cubic threefolds.
\newblock {\em Compos. Math.}, 143(4):1037--1049, 2007.

\bibitem{Markman_Torelli}
E.~Markman.
\newblock A survey of {T}orelli and monodromy results for
  holomorphic-symplectic varieties.
\newblock In {\em Complex and differential geometry}, volume~8 of {\em Springer
  Proc. Math.}, pages 257--322.

\bibitem{Markman_constraints}
E.~Markman.
\newblock Integral constraints on the monodromy group of the hyper{K}\"ahler
  resolution of a symmetric product of a {$K3$} surface.
\newblock {\em Internat. J. Math.}, 21(2):169--223, 2010.

\bibitem{Mongardi_CRAS}
G.~Mongardi.
\newblock On natural deformations of symplectic automorphisms of manifolds of
  {$K3^{[n]}$} type.
\newblock {\em C. R. Math. Acad. Sci. Paris}, 351(13-14):561--564, 2013.

\bibitem{Mongardi}
G.~Mongardi.
\newblock A note on the {K}\"ahler and {M}ori cones of hyperk\"ahler manifolds.
\newblock {\em Asian J. Math.}, 19(4):583--591, 2015.

\bibitem{Mukai}
S.~Mukai.
\newblock {\em An introduction to invariants and moduli}, volume~81 of {\em
  Cambridge Studies in Advanced Mathematics}.
\newblock Cambridge University Press, Cambridge, 2003.
\newblock Translated from the 1998 and 2000 Japanese editions by W. M. Oxbury.

\bibitem{Murre}
J.~P. Murre.
\newblock Algebraic equivalence modulo rational equivalence on a cubic
  threefold.
\newblock {\em Compositio Math.}, 25:161--206, 1972.

\bibitem{Nikulin}
V.~V. Nikulin.
\newblock Integral symmetric bilinear forms and some of their applications.
\newblock {\em Math. USSR Izv.}, 14:103--167, 1980.

\bibitem{VoisinCubic}
C.~Voisin.
\newblock Th\'eor\`eme de {T}orelli pour les cubiques de $\mathbf{P}^5$.
\newblock {\em Invent. Math.}, 86(3):577--601, 1986.

\end{thebibliography}

\end{document}